\documentclass[a4paper]{article}

\author{Jan Bouwe van den Berg\thanks{Department of Mathematics, Vrije Universiteit Amsterdam, The Netherlands, \href{mailto:janbouwe@few.vu.nl}{janbouwe@few.vu.nl}; partially supported by NWO-VICI grant 639033109},  Wouter Hetebrij\thanks{Department of Mathematics, Vrije Universiteit Amsterdam, The Netherlands, \href{mailto:w.a.hetebrij@vu.nl}{w.a.hetebrij@vu.nl}.} \ and Bob Rink\thanks{Department of Mathematics, Vrije Universiteit Amsterdam, The Netherlands, \href{mailto:b.w.rink@vu.nl}{b.w.rink@vu.nl}.}}
\title{More on the parameterization method for center manifolds}
\usepackage[english]{babel}

\usepackage{amsmath}
\usepackage{amssymb} 
\usepackage{amsthm}

\usepackage{enumitem}
\usepackage{mathtools}

\usepackage{hyperref}
\hypersetup{
	pdfborder={0 0 0},
	pdftoolbar = true,
	pdfmenubar = true,
}

\usepackage{amsrefs}
\usepackage[capitalize,noabbrev]{cleveref}
	\crefname{equation}{}{}
	\crefname{enumi}{}{}

\usepackage{multirow}
\usepackage{longtable}
\usepackage{arydshln}

\newcounter{Rest}

\newtheorem{theorem}[Rest]{Theorem}
\newtheorem*{theorem*}{Theorem}
\newtheorem{corollary}[Rest]{Corollary}
\newtheorem*{corollary*}{Corollary}
\newtheorem{lemma}[Rest]{Lemma}
\newtheorem*{lemma*}{Lemma}
\newtheorem{proposition}[Rest]{Proposition}
\theoremstyle{remark}
\newtheorem{remark}[Rest]{Remark}
\newtheorem*{remark*}{Remark}

\numberwithin{equation}{subsection}
\numberwithin{Rest}{section}

\newcommand{\supnorm}[1]{\| #1 \|_0}

\newcommand{\operatornorm}[1]{\| #1 \|_{\textup{op}}}
\newcommand{\norm}[1]{\| #1 \|}

\newcommand{\tmatrix}[3]{\left( \begin{smallmatrix} #1 \\ #2 \\ #3 \end{smallmatrix} \right)}

\newcommand{\isdef}{:=}
\newcommand{\isfed}{=:}

\newcommand{\ds}{\text{d}s}
\newcommand{\taumap}{time $\tau$-map}
\newcommand{\tmap}{time $t$-map}
\newcommand{\tmaps}{time $t$-maps}
\newcommand{\functional}[1]{\tau(#1)}
\newcommand{\functionalsub}[2]{\tau_{#1}(#2)}

\usepackage{subcaption}
\usepackage{dcpic,pictex}

\begin{document}

\maketitle

\begin{abstract}
In a previous paper, see \cite{Hetebrij19}, we generalized the parameterization method of Cabr\'{e}, Fontich and De la Llave to center manifolds of discrete dynamical systems. In this paper, we extend this result to several different settings. The natural setting in which center manifolds occur is at bifurcations in dynamical systems with parameters. Our first results will show that we can find parameter-dependent center manifolds near bifurcation points. 
Furthermore, we will generalize the parameterization method to center manifolds of fixed points of ODEs. Finally, we will apply our method to a reaction diffusion equation. In our application, we will show that the freedom to obtain the conjugate dynamics in normal form makes it possible to obtain detailed qualitative information about the center dynamics.
\end{abstract}

\section{Introduction}

The parameterization method introduced by Cabr\'{e}, Fontich and De la Llave in \cite{Cabre03,Cabre03-2,Cabre05} is used to find (un)stable manifolds associated to hyperbolic equilibria in dynamical systems. In a previous paper, \cite{Hetebrij19}, we gave a generalization of the parameterization method which can be used to find center manifolds at fixed points of  discrete dynamical systems. 

The original method for (un)stable manifolds has been applied to delay differential equations, see \cite{Groothedde17}, and partial differential equations, see \cite{Reinhardt19}. Furthermore, the method is useful for computational existence proofs of for example homoclinic and heteroclinic orbits, see \cites{VandenBerg11,Lessard14,VandenBerg15-2}. The method has also been used for constructing (un)stable manifolds of periodic orbits, see \cite{Castelli18}. Finally, the method was generalized in \cite{Baldoma16,Baldoma16-2} to find invariant manifolds for parabolic fixed points. 

One of the goals of this paper is to give a generalization of the parameterization method for center manifolds in systems with parameters. For a discrete dynamical system, $F : X \to X$ on a Banach space $X$, the parameterization method constructs a conjugacy $K$, the parameterization of the center manifold, between the center subspace  $X_c$ and the the center manifold, as well as a conjugate dynamical system $R : X_c \to X_c$ on the center subspace such that $K(X_c)$ lies tangent to the center subspace and the conjugacy equation
\begin{align}
F \circ K = K \circ R \label{ConjugayEquationIntroduction}
\end{align}
is satisfied. In other words, orbits of $R$ on the center subspace are mapped by $K$ to orbits of $F$ on the center manifold. 

To obtain $K$ and $R$ such that \cref{ConjugayEquationIntroduction} holds, in \cite{Hetebrij19} we first rewrite  \cref{ConjugayEquationIntroduction} as a fixed point problem of the form $(K,R) = \Theta(K,R)$. We prove that $\Theta$ is a contraction on a suitable function space to prove the existence of both $K$ and $R$. We refer the reader to \cite{Hetebrij19} for the details of this proof.

The first generalization we present here is finding a smooth parameter-dependent conjugacy for a dynamical system with parameters. Center manifolds naturally occur at bifurcations, and the change in dynamical behaviour before and after the bifurcation takes place on the center manifold. Thus finding a parameter-dependent center manifold at the bifurcation point allows us to quantitatively describe changes in dynamical behaviour. Using normal form theory, see for instance \cite{Chow94}, one can also obtain qualitative information about the dynamical behaviour near the bifurcation point. The main advantage of our method is that we have the freedom to obtain the Taylor approximation of the conjugate dynamical system in normal form and at the same time  obtain explicit bounds on the difference between the conjugate dynamical system and the normal form. Those explicit bounds allow us to obtain quantitative information about the dynamical behaviour near the bifurcation point. In the proof, we will extend the original dynamical system by adding the parameters as new variables. The main challenge in the proof of this generalization consists of choosing the right norm on the extended Banach space $\Lambda \times X$, where $\Lambda$ is our parameter space.

We also want to generalize our method to center manifold in ODEs, i.e\ continuous time dynamical systems. For continuous time dynamical systems given by an ODE $\dot{x} = f(x)$, we want to construct a conjugacy $K$ between the center subspace and the center manifold as well as a vector field $\dot{x} = R(x)$ on the center subspace such that orbits of $R$ are mapped by $K$ to orbits of $f$, analogous to what we did for discrete dynamical systems.  This means that if we take a solution $y(t)$ of the ODE $\dot{x} = R(x)$ on the center subspace, $K(y(t))$ should be a solution of the ODE $\dot{x} = f(x)$ on the center manifold, i.e. \ 
\begin{align}
 f(K(y(t)) = DK(y(t)) \cdot R(y(t)). \label{IntroductionEquation}
\end{align}
The main difference with the conjugacy equation \cref{ConjugayEquationIntroduction} of the discrete case, is the spatial derivative of $K$ at the right hand side of \cref{IntroductionEquation}. In the discrete case we rewrote \cref{ConjugayEquationIntroduction} as a fixed point problem, $(K,R) = \Theta(K,R)$, and showed that $\Theta$ is a contraction. If we would use the same strategy in the continuous case, we would try to rewrite \cref{IntroductionEquation} as a fixed point problem $(K,R) = \Theta(K,R)$. However, we now encounter two problems we have to overcome. The center manifold, and therefore $K$, that we obtain will only be $C^n$, whereas $\Theta$ is a differential operator. Thus we must find a suitable function space $A$ on which $\Theta$ is well-defined, as well as define a norm on $A$ and find a set $B \subset A$ such that $\Theta$ is a contraction on $B$. To circumvent those two problems, we will solve an equivalent problem to obtain $K$ and $R$.

If $K$ and $R$ solve \cref{IntroductionEquation}, then $K$ is also a conjugacy between the \tmap\ on the center subspace and the \tmap\ on the center manifold, for any $t \ge 0$, where we will denote the latter \tmap\ by $\varphi_t$. Conversely, if we find a single conjugacy $K$ and a semigroup of discrete dynamical system $\{\psi_t\}_{t \ge 0}$ on the center subspace such that 
\begin{align*}
K \circ \psi_t = \varphi_t \circ K \qquad \text{ for all } t \ge 0,
\end{align*}
then we choose $R$ to be the infinitesimal generator of $\{ \psi_t\}_{t\ge 0}$ and $K$ maps orbits on the center subspace to orbits on the center manifold., i.e.\ $K$ and $R$ solve \cref{IntroductionEquation} The main advantage of this equivalence is that the \tmap\ of $f$ is a discrete dynamical system, for which we can use our parameterization method from \cite{Hetebrij19}. On the other hand, instead of finding a conjugacy and a conjugate dynamical system for one dynamical system, we need to find a collection of conjugate dynamical systems and a single conjugacy (independent of $t$) for a collection of dynamical systems.  Hence we will prove that the conjugacy we obtain for a single \tmap\ of $f$ will be a conjugacy for \emph{all} $t >0$ for a judicious choice of $\psi_t$.

We will demonstrate (the strength of) our method with an application which is simple yet illustrates the essential steps involved. For our application, we consider the spatial dynamics of a reaction diffusion equation on a 2D grid. By varying a parameter in the system, a period-doubling bifurcation occurs. Using the freedom to obtain the conjugate dynamics in normal form, and by computing explicit bounds on the difference between the normal form and the conjugate dynamics, we obtain explicit regions in the phase space in which a period-2 orbit must lie after the bifurcation. To obtain the explicit error bounds and subsequently the explicit regions of validity, we use the Mathematica Notebook available in \cite{Mathematica}. Furthermore, using the explicit error bounds, we also prove that a heteroclinic orbit emerges between the period-2 orbit and the stationary point.

\subsubsection*{Outline of the paper}

The paper consists of three parts. We first introduce notation and restate our main theorem from \cite{Hetebrij19}. In \cref{SectionParameters} we will give the generalization of the parameterization method to dynamical systems with parameters. In addition to this first generalization, we will also show how one can compute explicit error bounds for the Taylor approximations of $K$ and $R$. We then continue in \cref{SectionODE} with the generalization to center manifolds in continuous time dynamical systems. Our proof will use the equivalent problem of finding a conjugate dynamical systems for the \tmaps\ and a single conjugacy, which is done using multiple intermediate steps. Finally, we conclude the paper in \cref{ExampleSection} with the application of the parameterization method to obtain a period doubling-bifurcation in a reaction diffusion equation.

\subsection{Notation and conventions}\label{SectionNotation}

We use the following notation and conventions in this paper.
\begin{itemize}
\item For functions $f: X \to Y$ between Banach spaces, we denote with
\begin{align*}
\norm{f}_n \isdef \max_{0 \le m \le n} \sup_{x \in X} \norm{D^mf(x)}
\end{align*}
the $C^n$ norm of $f$ for $n \ge 0$.

For $X$ and $Y$ Banach spaces, we denote with 
\begin{align*}
C^n_b(X,Y) &\isdef \left\{ f  : X \to Y \mid f \text{ is } C^n \text{ and } \norm{f}_n < \infty \right\}.
\end{align*}
the Banach space of all $C^n$ bounded functions between $X$ and $Y$.

\item For a bounded linear operator $A: X \to Y$ between Banach spaces, we denote with
\begin{align*}
\operatornorm{A} \isdef \sup_{\norm{x} =1} \norm{Ax}
\end{align*}
the operator norm of $A$.

For $X$ and $Y$ Banach spaces, we denote with 
\begin{align*}
\mathcal{L}(X,Y) \isdef \left\{ A: X \to Y \mid A \text{ is a linear operator and } \operatornorm{A} < \infty \right\}
\end{align*}
the Banach space of all bounded linear operators between $X$ and $Y$.

\item For an unbounded linear operator $A : \mathcal{D}(A) \subset X \to Y$ between Banach spaces, we denote with $\mathcal{D}(A)$ its domain, i.e.\ $x \in \mathcal{D}(A)$ if and only if $Ax$ exists and lies in $Y$.

Furthermore, we call an unbounded linear operator $A$ sectorial if it satisfies the following three properties.
\begin{itemize}
\item The operator $A$ is closed, i.e.\ the graph of $A$ is closed in $X \times Y$, and densely defined, i.e.\ $\mathcal{D}(A)$ is dense in $X$.
\item There exist constants $w \in \mathbb{R}$ and $\lambda > 0$ such that the spectrum $\sigma(A)$ is contained in the sector $\left\{ z \in \mathbb{C} \ \middle| \ |\operatorname{Im}(z)| < \lambda \operatorname{Re}(w - z) \right\}$.
\item There exists a constant $C > 0$ such that for all $z \in \mathbb{C}$  outside the sector $\left\{ z \in \mathbb{C} \ \middle| \ |\operatorname{Im}(z)| < \lambda \operatorname{Re}(w - z) \right\}$ the linear operator $(A - z \operatorname{Id})^{-1}$ is bounded by $C/ |z - w|$.
\end{itemize}

\item Let $\varepsilon > 0$ and $U \subset \mathbb{R}^m$. We denote with $U^\varepsilon \isdef  \left\{ x \in \mathbb{R}^m \ \middle| \ \text{dist}(x,U) < \varepsilon \right\}$ the $\varepsilon$-neighborhood of $U$.

\end{itemize}

\subsection{Parameterization theorem for center manifolds}

For the sake of completeness, we will repeat the statement of the paramaterization method for center manifolds for discrete systems in \cite{Hetebrij19}.

\begin{theorem}[Parameterization of the center manifold]\label{MainTheorem} Let $X$ be a Banach space and $F{}{}: X \to X $ a  $C^n$, $n \ge 2$, discrete dynamical system on $X$ such that $0$ is a fixed point of $F{}{}$. Denote $F{}{} = A + g{}{}$ with $A \isdef Df(0)$ and let $k_c : X_c \to X_c$ be chosen. Assume that
\begin{enumerate}
\item\label{MainTheoremSubpaces } There exist closed $A$-invariant subspaces $X_c$, $X_u$ and $X_s$ such that $X = X_c \oplus X_u \oplus X_s$. We write $A = \tmatrix{A_c & 0 & 0}{0 & A_u & 0}{0 & 0 & A_s}$ where we define $A_c \isdef A \big|_{X_c}$, and similarly define $A_u$ and $A_s$.
\item[1a.] The norm on $X = X_c \oplus X_u \oplus X_h$ satisfies
\begin{align}
\norm{x} = \max \left\{ \norm{x_c}_c,\norm{x_u}_u,\norm{x_s}_s \right\}, \qquad \text{ for } x = (x_c,x_u,x_s) \label{DesiredPropertyNorm}
\end{align}
where $x_{i} \in X_{i}$ and $\norm{\cdot}_{i}$ is the norm on $X_{i}$ for $i = c,u,s$.
\item\label{MainTheoremOperators } The linear operators $A_c$ and $A_u$ are invertible.
\item\label{MainTheoremNorm } The norm on $X$ is such that
\begin{align*}
\operatornorm{A_c^{-1} }^{\tilde{n}} \operatornorm{ A_s } < 1 && \text{ and } && \operatornorm{A_u^{-1} } \operatornorm{A_c}^{\tilde{n}} < 1 && \text{ for all }  1 \le \tilde{n} \le n.
\end{align*}
\item\label{MainTheoremKnownBounds } The non-linearities $g$ and $k_c$ satisfy
\begin{align*}
g{}{} &\in \left\{ h \in C^{n}_b(X,X) \mid h(0)=0, \   Dh(0) = 0 \text{ and }\supnorm{Dh} < L_g \right\}, \\ 
k_c &\in \left\{ h \in C^{n}_b(X_c,X_c) \mid h(0) = 0 , \ Dh(0) = 0 \text{ and } \supnorm{Dh} < L_c \right\},
\end{align*}
for $L_g$ and $L_c$ small enough, as defined in Remark 2.4 of \cite{Hetebrij19}.
\end{enumerate}
Then there exist a $C^{n}$ conjugacy $K{}{} :  X_c \to X$ and $C^{n}$ discrete dynamical system $R = A_c + r : X_c \to X_c$ such that
\begin{align}
(A + g{}{}) \circ K{}{} = K{}{} \circ (A_c +r) \label{ConjugacyEquation}.
\end{align}
Furthermore, $A_c + r$ is globally invertible and $K{}{} =  \iota +  \tmatrix{k_c}{k_u}{k_s}$ with $\iota: X_c \to X$ the inclusion map.
\end{theorem}
\begin{remark}\label{MainTheoremEquivalentNorm}
If $F{}{}$ is $C^n$ with respect to an arbitrary norm on $X$ and condition \cref{MainTheoremSubpaces } is satisfied, we can define an equivalent norm on $X$ which satisfies condition 1a and leaves the norm unchanged on $X_c$, $X_u$ and $X_s$. The reason for asking condition 1a is that it allows us to make our estimates explicit.
\end{remark}
\begin{remark}The original version of \cref{MainTheorem} states condition \cref{MainTheoremKnownBounds } as: 
\begin{itemize}
\item[\cref{MainTheoremKnownBounds }.]The non-linearities $g$ and $k_c$ satisfy
\begin{align*}
g{}{} &\in \left\{ h \in C^{n}_b(X,X) \mid h(0)=0, \   Dh(0) = 0 \text{ and }\supnorm{Dh} \le L_g \right\}, \\ 
k_c &\in \left\{ h \in C^{n}_b(X_c,X_c) \mid h(0) = 0 , \ Dh(0) = 0 \text{ and } \supnorm{Dh} \le L_c \right\},
\end{align*}
for $L_g$ and $L_c$ small enough.
\end{itemize}
Hence it seems like our statement of the theorem is slightly weaker than in \cite{Hetebrij19}. However $L_g$ and $L_c$ are defined in terms of strict inequalities. Thus we can replace the less than or equal to signs in condition 4 with strict inequalities, if we also replace the strict inequalities in the definition of $L_g$ and $L_c$ with less than or equal to signs.
\end{remark}

\begin{remark}
For a general dynamical system, the non-linearity $g$ will be unbounded in $C^n$. To satisfy the fourth assumption, the usual approach in $\mathbb{R}^n$ is to multiply $g$ with a smooth cut-off function $\xi$, such that $g \xi$ is bounded in $C^n$. By shrinking the support of the cut-off function, the norm of the derivative of $g \xi$ can then be made as small as desired. If we apply \cref{MainTheorem} to the map $x \mapsto Ax + g(x) \xi(x)$, we find a local center manifold for our original system, which is valid on the region where $\xi \equiv 1$. In \cref{ExampleSection} we will see a different trick to bound the derivative of $g$.
\end{remark}

\section{Parameter-dependent systems}\label{SectionParameters}

Center manifolds naturally occur at bifurcations, which can happen in dynamical systems with parameters. Hence a natural setting to which we want to generalize \cref{MainTheorem}, are discrete dynamical systems with parameters. For simplicity, we will assume without loss of generality that for $\lambda = 0$ there exists a center manifold.

\begin{theorem}\label{ParametersTheorem} Let $X$ be a Banach space, $W \subset \mathbb{R}^m$ open and $F{}{}: W \times X \to X $ a  jointly $C^n$, $n \ge 2$, discrete dynamical system on $X$ with parameter space $0 \in W$ such that $0$ is a fixed point of $F{}{}_0(\cdot) \isdef F{}{}(0,\cdot)$. Denote $F{}{}(\lambda,x) = Ax + C\lambda + g_\lambda{}{}$ with $A \isdef D_xF{}{}(0,0)$ and $C \isdef D_\lambda F{}{}(0,0)$. Let $k_c : W \times X_c \to X_c$ be chosen such that $k_{c}(0,0) = 0$ and $D_\lambda k_c(0,0) = 0$. Assume that $F_0$ satisfies the conditions of \cref{MainTheorem} where we choose $k_{c,0}(\cdot) \isdef k_c(0,\cdot) : X_c \to X_c$, with condition \cref{MainTheoremNorm } replaced by:
\begin{itemize}
\item[\cref{MainTheoremNorm }a.] The norm on $X$ is such that
\begin{align*}
\max\{1,\operatornorm{A_c^{-1} }\}^{n} \operatornorm{ A_s } < 1 && \text{ and } && \operatornorm{A_u^{-1} } \max\{1,\operatornorm{A_c}\}^{n} < 1.
\end{align*}
\end{itemize}
Furthermore, we assume that
\begin{itemize}
\item[5.] Let $\tilde{W} \subset W$ and $\varepsilon > 0$ be such that $\tilde{W}^\varepsilon \subset W$ and 
\begin{align*}
g &\in \left\{ h \in C_b^n(\tilde{W}^\varepsilon \times X, X) \  \middle| \ \sup_{\lambda \in \tilde{W}^\varepsilon} \supnorm{D_x h(\lambda, \cdot)} < L_g \right\}, \\
k_c &\in \left\{ h \in C_b^n(\tilde{W}^\varepsilon \times X, X) \  \middle| \  \sup_{\lambda \in \tilde{W}^\varepsilon} \supnorm{D_x h(\lambda, \cdot)} < L_c \right\},
\end{align*}
for the $L_g$ and $L_c$ from \cref{MainTheorem} for $F_0$.
\end{itemize}
Then there exist a jointly $C^{n}$ conjugacy $K{}{} : \tilde{W} \times X_c \to X$ and jointly $C^{n}$ discrete dynamical system $R_\lambda = A_c + r_\lambda : \tilde{W} \times X_c \to X_c$ such that
\begin{align}
(A + g_\lambda{}{}) \circ K_\lambda{}{} = K_\lambda{}{} \circ (A_c +r_\lambda) \label{ConjugacyEquationParameter}.
\end{align}
for all $\lambda \in \tilde{W}$. Furthermore, $A_c + r_\lambda$ is globally invertible and $K_\lambda{}{} =  \iota +  \tmatrix{k_{c,\lambda}}{k_{u,\lambda}}{k_{s,\lambda}}$ with $\iota: X_c \to X$ the inclusion map for all $\lambda \in \tilde{W}$ .
\end{theorem}
\begin{remark}\label{RemarkBoundednessNorm}
All norms on $\mathbb{R}^m$ are equivalent, and thus if $F: W \times X \to X$ is smooth with respect to any norm on $\mathbb{R}^m$, it is smooth with respect to all norms on $\mathbb{R}^m$. Similarly, if  $F : W \times X \to X$ is bounded in $C^n$ with respect to any norm on $\mathbb{R}^m$, it is bounded with respect to all norms. However, the $C^n$ norm of $F$ does depend on the norm of $\mathbb{R}^m$.
\end{remark}

\begin{remark}\label{SpectrumParameters}
From condition 3a it follows that neither $A_u$ nor $A_s$ contains $1$ in their spectrum. That means that $\operatorname{Id} - A_u$ and $\operatorname{Id} - A_s$ are both invertible.  In particular, we have for $(x_c,x_u,x_s) \in X$
\begin{align*}
\begin{pmatrix} x_c \\ x_u \\ x_s \end{pmatrix} + A \begin{pmatrix} 0 \\ ( \operatorname{Id} - A_u )^{-1} x_u  \\ ( \operatorname{Id} - A_s )^{-1} x_s \end{pmatrix} &= \begin{pmatrix} 0 \\ ( \operatorname{Id} - A_u )^{-1} x_u  \\ ( \operatorname{Id} - A_s )^{-1} x_s \end{pmatrix}  + 
\begin{pmatrix} x_c \\0 \\ 0 \end{pmatrix}.
\end{align*}
That is, for every $z \in X$, we can find  $y \in X_c$ and $x \in X_u \oplus X_s$ such that $z + Ax = x+ y$.
\end{remark}

\begin{proof}
The  idea behind the proof is to create a dynamical system $\tilde{F}{}{}: W \times X \to W \times X$ by considering the parameters as extra variables, and apply \cref{MainTheorem} to this new dynamical system. However, if $W \subsetneq \mathbb{R}^m$, \cref{MainTheorem} cannot be applied as $W \times X$ is not a Banach space. So we first have to extend $\tilde{F}{}{}$ to a dynamical system on $\mathbb{R}^m \times X$. 

Let $\varepsilon > 0$ and $\tilde{W}$ be such that condition $5$ is satisfied.  We consider a smooth cut-off function $\xi : \mathbb{R}^m \to [0,1]$ such that $\xi \equiv 1$ on $\tilde{W}$, and its support lies in $\tilde{W}^\varepsilon$. Consider the dynamical system
\begin{align}
\tilde{F}{}{} : Y \isdef \mathbb{R}^m \times X \to Y, (\lambda, x ) \mapsto \begin{pmatrix} \operatorname{Id} & 0 \\ C & A \end{pmatrix} \begin{pmatrix} \lambda \\ x \end{pmatrix} + \begin{pmatrix} 0 \\  g_\lambda(x) \xi(\lambda) \end{pmatrix}, \label{NormParameterSystemEq}
\end{align}
with the convention that $g_\lambda(x) \xi(\lambda) = 0$ outside $W$. 
In particular we have that $\tilde{F}{}{} \equiv (\operatorname{Id} , F_\lambda{}{} )$ on $\tilde{W} \times X$.

To make $Y$ a Banach space, we have to define a norm on $Y$. Denote the norm on $\mathbb{R}^m$ with $\norm{\cdot}_{\mathbb{R}^m}$, which we can choose freely, and  the norm on $X$ with  $\norm{\cdot}_X$. We equip $Y$ with the supremum norm $\norm{(\lambda,x)}_{\tilde{Y}} \isdef \max\{ \norm{\lambda}_{\mathbb{R}^m} , \norm{x}_X\}$.  We now want to apply \cref{MainTheorem} to $\tilde{F}{}{}$. Hence we want to use the norm $\norm{\cdot}_{\tilde{Y}}$ to check the smoothness of $\tilde{F}{}{}$ and find the invariant subspaces $Y_c$, $Y_u$ and $Y_s$ as in assumption \cref{MainTheoremSubpaces } of \cref{MainTheorem}. We then define an equivalent norm $\norm{\cdot}_{Y}$ on $Y$ such that condition 1a of \cref{MainTheorem} is satisfied, see \cref{RightNormEquation}.
\ \\

\textit{\underline{Smoothness and linearization:}} We first have to show that $\tilde{F}{}{}$ is $C^n$. The linear part of $\tilde{F}{}{}$ is smooth, and the non-linear part is $C^n$, since $g_\lambda :W \times X \to X$ is jointly $C^n$ by assumption and $\xi$ is smooth by construction. Thus their product is also $C^n$. Furthermore, $(0,0)$ is a fixed point of $\tilde{F}{}{}$ and the derivative of $\tilde{F}{}{}$ at $(0,0)$ is given by
\begin{align}
\tilde{A} \isdef D\tilde{F}{}{}(0,0) = \begin{pmatrix} \operatorname{Id} & 0 \\ C & A\end{pmatrix},
\end{align}
since $g(0,0) = 0$, $D_\lambda g(0,0) = 0$ and $D_x g(0,0) = 0$. In particular we see that our non-linearity is indeed $\tilde{g}(\lambda,x) \isdef (0, g_\lambda(x)  \xi(\lambda))$.
\ \\

\textit{\underline{Invariant subspaces:}} For the first condition of \cref{MainTheorem} we  need to define closed $\tilde{A}$-invariant subspaces $Y_c$, $Y_u$ and $Y_s$.  Let $(e_i)_{i=1}^m \subset \mathbb{R}^m$ be a basis of $\mathbb{R}^m$, and denote $C_i = C e_i$. By \cref{SpectrumParameters}, we can find  elements $(x_i)_{i=1}^m \subset X_u \oplus X_s$ and $(y_i)_{i=1}^m \subset X_c$ such that $C_i + Ax_i  = x_i + y_i$. We  define the closed subspaces
\begin{align}
Y_c &\isdef \operatorname{span}\left( \left\{ (0,x) \in Y \mid x \in X_c \right\}, \left\{ (e_i,x_i) \mid 1 \le i \le m \right\} \right), \label{EquationParameterSubspaceCenter} \\
Y_u &\isdef \left\{ (0,x) \in Y \mid x \in X_u \right\}, \\
Y_s &\isdef \left\{ (0,x) \in Y \mid x \in X_s \right\}.
\end{align}
Since $X_c \oplus X_u \oplus X_s = X$, we have $Y_c \oplus Y_u \oplus Y_s = \mathbb{R}^m \times X = Y$. Furthermore, by construction we have that $Y_u$ and $Y_s$ are invariant under $\tilde{A}$, and for an element $y = (0,x) + \sum_{i=1}^m \mu_i (e_i,x_i) \in Y_c$, with $x \in X_c$, we have
\begin{align*}
\tilde{A} y = \begin{pmatrix} \operatorname{Id} & 0 \\ C & A\end{pmatrix} \begin{pmatrix} 0 \\ x \end{pmatrix} +  \sum_{i=1}^m \mu_i \begin{pmatrix} \operatorname{Id} & 0 \\ C & A\end{pmatrix} \begin{pmatrix} e_i \\ x_i \end{pmatrix} = \begin{pmatrix} 0 \\ Ax \end{pmatrix} + \sum_{i=1}^m \mu_i \begin{pmatrix} e_i \\ C_i + Ax_i \end{pmatrix}.
\end{align*}
By definition, $X_c$ is invariant under $A$, thus $(0,Ax) \in Y_c$. By construction, we have $(e_i, C_i + Ax_i) = (e_i, x_i + y_i) = (e_i,x_i) + (0,y_i)$. As $y_i \in X_c$, we have $(e_i,x_i) + (0,y_i) \in Y_c$. Now that we have found our invariant subspaces, we can define a norm $\norm{\cdot}_Y$ on $Y$ which is equivalent with $\norm{\cdot}_{\tilde{Y}}$ and satisfies condition 1a of \cref{MainTheorem}:
\begin{align}
\norm{(\lambda,x+\functional{\lambda})}_Y \isdef \max\{ \norm{\lambda}_{\mathbb{R}^m}, \norm{x}_X \}, \label{RightNormEquation}
\end{align}
where $\functional{\lambda} = \sum_{i=1}^m \lambda_i x_i$ when $\lambda = \sum_{i=1}^m \lambda_i e_i$. In particular, we have that $\tilde{A}_u = A_u$ and $\tilde{A}_s = A_s$, and their operator norms are equal given the norm $\norm{\cdot}_{Y}$ on $Y$.
\ \\

\textit{\underline{Invertibility:}} For the second condition of \cref{MainTheorem}, we have to check that $\tilde{A}$ restricted to $Y_c$ or $Y_u$ is invertible. As we already noted, $\tilde{A}_u = A_u$, which is invertible. One can check that the inverse of $\tilde{A}_c$ is defined by
\begin{align*}
\begin{cases}
\tilde{A}_c^{-1} \begin{pmatrix} 0 \\ x \end{pmatrix} = \begin{pmatrix} 0 \\ A_c^{-1} x \end{pmatrix} & \text{ for } x \in X_c,\\
\tilde{A}_c^{-1} \begin{pmatrix} e_i \\ x_i \end{pmatrix} = \begin{pmatrix} e_i \\ x_i \end{pmatrix} - \begin{pmatrix} 0 \\ A_c^{-1} y_i \end{pmatrix} & 1 \le i \le m.
\end{cases}
\end{align*}
Hence we have to show that $\tilde{A}_c^{-1}$ defines a bounded linear map. We  want to reiterate that we can choose the norm on $\mathbb{R}^m$ since all norms on $\mathbb{R}^m$ are equivalent. Let $M > 0$, and define 
\begin{align}
\left\| \sum_{i=1}^m \lambda_i e_i \right\|_{\mathbb{R}^m} \isdef M \sum_{i=1}^m | \lambda_i| . \label{NormOnRn}
\end{align}
We will not give an explicit value for $M$, but we will argue during the proof that the conditions of \cref{MainTheorem} are satisfied for $M$ large enough. The norm of $y = (0,x) + \sum_{i=1}^m \lambda_i (e_i,x_i)$ is given by
\begin{align*}
\norm{y}_Y = \max\left\{ M \sum_{i=1}^m |\lambda_i |, \left\| x \right\|_X  \right\}.
\end{align*}
Let $y \in Y_c$ be inside the unit ball, i.e.\ $M \sum_{i=1}^m |\lambda_i | \le 1$ and  $\left\| x \right\|_X \le 1$.
Then we estimate
\begin{align}
\norm{\tilde{A}_c^{-1} y}_Y &= \left\|  \begin{pmatrix} \sum_{i=1}^m \lambda_i e_i \\ A_c^{-1} x + \sum_{i = 1}^m \lambda_i x_i -  \sum_{i = 1}^m \lambda_i A_c^{-1} y_i \end{pmatrix} \right\|_Y \nonumber \\
		&= \max\left\{  M \sum_{i=1}^m  |\lambda_i| , \left\| A_c^{-1} x - \sum_{i=1}^m \lambda_i  A_c^{-1} y_i  \right\|_X \right\} \nonumber \\
		&\le \max\left\{ 1 , \operatornorm{A_c^{-1}} \norm{x}_X + \sum_{i \in I }^m |\lambda_i| \operatornorm{A_c^{-1}} \left\| y_i \right\|_X   \right\} \nonumber \\
		&\le \max\left\{ 1 , \operatornorm{A_c^{-1}} + M^{-1} \operatornorm{A_c^{-1}} \mathcal{C}_y  \right\} . \label{NormParameterInverseEq1}
\end{align}
Here we define $\mathcal{C}_y \isdef  \max_{1 \le i \le m} \left\{ \left\| y_i \right\|_X \right\} $ in the last inequality, and we use that $\sum_{i=1}^m |\lambda_i | \le 1 / M$. Thus $\tilde{A}_c^{-1}$ is a bounded linear operator, and hence $\tilde{A}_c$ is invertible, which shows that the second condition of \cref{MainTheorem} is satisfied.
\ \\

\textit{\underline{Bounds on the linearization:}} For the third condition of \cref{MainTheorem}, we have to estimate the operator norm of $\tilde{A}_c$, as we have already computed a bound for the operator norm of $\tilde{A}_c^{-1}$. For an element $y \in Y_c$ in the unit ball, we have
\begin{align}
\norm{\tilde{A}_c y}_Y &= \left\| \begin{pmatrix} \sum_{i=1}^m \lambda_i e_i \\ A_c x + \sum_{i=1}^m  \lambda_i x_i + \sum_{i=1}^m  \lambda_i y_i \end{pmatrix}\right\|_Y \nonumber \\
			&= \max \left\{ M \sum_{i=1}^m | \lambda_i| , \left\| A_c x + \sum_{i=1}^m  \lambda_i   y_i  \right\|_X \right\} \nonumber \\
			&\le \max \left\{ 1 , \operatornorm{A_c} \norm{x}_X + \sum_{i=1}^m  |\lambda_i| \norm{ y_i}_X \right\} \nonumber \\
			&\le \max \left\{ 1 , \operatornorm{A_c}  + M^{-1} \mathcal{C}_y \right\}. \label{NormParameterNormsEq1}
\end{align}
Since we estimate both $\operatornorm{\tilde{A}_c}$ and $\operatornorm{\tilde{A}_c^{-1}}$ by at least $1$, it is enough to check the third assumption for $\tilde{n} = n$. From assumption 3a we have
\begin{align*}
\max\{1,\operatornorm{A_c^{-1} }\}^{n} \operatornorm{ A_s } < 1 && \text{ and } && \operatornorm{A_u^{-1} } \max\{1,\operatornorm{A_c}\}^{n} < 1.
\end{align*}
Therefore, there exists an $\varepsilon > 0$ such that also
\begin{align}
\max\{1,\operatornorm{A_c^{-1}} + \varepsilon\}^{n} \operatornorm{ A_s } < 1 && \text{ and } && \operatornorm{A_u^{-1} } \max\{1,\operatornorm{A_c} + \varepsilon \}^{n} < 1. \label{NormParameterNormsEq3}
\end{align}
In particular, we take $M$ big enough, such that from \cref{NormParameterNormsEq1,NormParameterInverseEq1} we obtain
\begin{align}
\operatornorm{\tilde{A}_c^{-1}} \le \max\left\{ 1 , \operatornorm{A_c} + \varepsilon\right\} && \text{ and } && \operatornorm{\tilde{A}_c} \le \max\left\{ 1 , \operatornorm{A_c} + \varepsilon \right\}. \label{NormParameterNormsEq4}
\end{align}
Together with the fact that $\operatornorm{\tilde{A}_u^{-1}} = \operatornorm{A_u^{-1}}$ and $\operatornorm{\tilde{A}_s} = \operatornorm{A_s}$, we obtain from \cref{NormParameterNormsEq3,NormParameterNormsEq4}
\begin{align*}
\operatornorm{\tilde{A}_c^{-1} }^{\tilde{n}} \operatornorm{ \tilde{A}_s } < 1 && \text{ and } && \operatornorm{\tilde{A}_u^{-1} } \operatornorm{\tilde{A}_c}^{\tilde{n}} < 1 && \text{ for all }  1 \le \tilde{n} \le n,
\end{align*}
where we used that $\operatornorm{\tilde{A}_c^{-1}}^{\tilde{n}} \le \operatornorm{\tilde{A}_c^{-1}}^n$ and $\operatornorm{\tilde{A}_c}^{\tilde{n}} \le \operatornorm{\tilde{A}_c}^n$ for all $1 \le \tilde{n} \le n$. Thus also the third assumption of \cref{MainTheorem} is satisfied.
\ \\

\textit{\underline{Bounds on the non-linearities:}} Finally, we have to check that fourth condition of \cref{MainTheorem} is satisfied for $\tilde{F}{}{}$. Recall that the non-linearity of $\tilde{F}{}{}$ is given by 
\begin{align*}
\tilde{g}(\lambda,x) = (0,g_\lambda(x) \xi(\lambda)),
\end{align*}
where $\xi: \mathbb{R}^m \to [0,1]$ is a cut-off function with its support on $\tilde{W}^\varepsilon$. Furthermore, we choose 
\begin{align*}
\tilde{k}_c : Y_c \to Y_c, (0,x) + \sum_{i=1}^m \lambda_i (e_i,x_i) \mapsto (0 , k_{c,\lambda}(x)\xi(\lambda)),
\end{align*}
 with $\xi$ the same cut-off function and $\lambda = \sum_{i=1}^m \lambda_i e_i$. We will only check that $\supnorm{D \tilde{g}} < L_g$, as checking $\supnorm{D\tilde{k}_c} < L_c$ will go analogously. We have
\begin{align}
\supnorm{D \tilde{g} } = \sup_{\substack{\lambda \in \mathbb{R}^m \\ x \in X }} \operatornorm{D\tilde{g}(\lambda,x)} =  \sup_{\substack{\lambda \in \mathbb{R}^m \\ x \in X }} \left\| \begin{pmatrix} 0 & 0 \\  B(\lambda,x) & D_x g_\lambda(x) \xi(\lambda) \end{pmatrix}  \right\|_\text{op}, \label{NormParameterNonLinearEq2}
\end{align}
where we introduce 
\begin{align}
B(\lambda,x) \isdef D_\lambda \tilde{g} (\lambda,x) =  \xi(\lambda) D_\lambda g(\lambda,x) + g_\lambda(x) \frac{\partial \xi}{\partial \lambda}(\lambda). \label{NormParameterNonLinearEq1}
\end{align}
Recall that $\xi \in [0,1]$ has support on $W^\varepsilon$ and we defined $\functional{\mu} = \sum_{i=1}^m \mu_i x_i$ for $\mu = \sum_{i=1}^m \mu_i e_i$. We estimate \cref{NormParameterNonLinearEq2} by
\begin{align}
\supnorm{D \tilde{g} } &=  \sup_{\substack{\lambda \in \mathbb{R}^m \\ x \in X }} \sup_{\substack{\norm{\mu}_{\mathbb{R}^m} \le 1 \\ \norm{y-\functional{\mu}}_X \le 1}} \left\| \begin{pmatrix} 0 & 0 \\  B(\lambda,x) & D_x g_\lambda(x) \xi(\lambda) \end{pmatrix} \begin{pmatrix} \mu \\ y \end{pmatrix} \right\|_Y \nonumber \\
		&\le \sup_{\substack{\lambda \in \tilde{W}^\varepsilon \\ x \in X }} \sup_{\norm{\mu}_{\mathbb{R}^m} \le 1 } \left\|   B(\lambda,x) \mu \right\|_X +  \sup_{\substack{\lambda \in \tilde{W}^\varepsilon \\ x \in X }} \sup_{\substack{\norm{\mu}_{\mathbb{R}^m} \le 1 \\ \norm{y-\functional{\mu}}_X \le 1}} \left\|D_x g_\lambda(x) \xi(\lambda) y   \right\|_X \nonumber \\
		&\le \sup_{\substack{\lambda \in \tilde{W}^\varepsilon \\ x \in X }} \sup_{\norm{\mu}_{\mathbb{R}^m} \le 1 } \left\|   B(\lambda,x) \mu \right\|_X +  \sup_{\lambda \in \tilde{W}^\varepsilon } \supnorm{D_x g_\lambda} \left( 1 + M^{-1} \mathcal{C}_x \right), \label{NormParameterNonLinearEq3} 
\end{align}
where we used the triangle inequality in the last line to estimate 
\begin{align*}
\norm{y}_X \le \norm{y - y(\mu)}_X + \norm{y(\mu)}_X \le 1 + \sum_{i=1}^m |\mu_i | \norm{x_i}_X \le 1 + M^{-1} \mathcal{C}_x.
\end{align*}
for all $\mu \in \mathbb{R}^m$ and $y \in X$ such that $\norm{ y - \functional{\mu}}_X \le 1$ and $\norm{\mu}_{\mathbb{R}^m} \le 1$. Here we introduce the constant $\mathcal{C}_x \isdef \max_{1\le i\le m} \{ \norm{x_i}_X \}$. From assumption $5$ we have that $\sup_{\lambda \in \tilde{W}^\varepsilon} \supnorm{D_x g_\lambda} < L_g$. Thus we must show that we can make $\operatornorm{B(\lambda,x)}$ arbitrary small uniformly on $\tilde{W}^\varepsilon \times X$. From the definition of $B$ in \cref{NormParameterNonLinearEq1} we obtain
\begin{align}
 \sup_{\substack{\lambda \in \tilde{W}^\varepsilon \\ x \in X }} & \sup_{\norm{\mu}_{\mathbb{R}^m} \le 1 } \left\|   B(\lambda,x) \mu \right\|_X  \nonumber \\
 			&\le  \sup_{\substack{\lambda \in \tilde{W}^\varepsilon \\ x \in X }} \sup_{\norm{\mu}_{\mathbb{R}^m} \le 1 }  \left\|  \xi(\lambda) D_\lambda g(\lambda,x) \mu \right\|_X  + \left\| g_\lambda(x) \frac{\partial \xi}{\partial \lambda}(\lambda) \mu \right\|_X  \nonumber \\
 			&\le  \sup_{\substack{\lambda \in \tilde{W}^\varepsilon \\ x \in X}} \left\| D_\lambda g(\lambda,x) \right\|_{\text{op}} + \sup_{\substack{\lambda \in \tilde{W}^\varepsilon \\ x \in X}} \norm{g(\lambda,x) } \left\|  \frac{ \partial \xi}{\partial \lambda} \right\|_{\text{op}}.
 \label{NormParameterNonLinearEq5}
\end{align} 
Here we used that $\xi (\lambda) \in [0,1]$ for $\lambda \in \tilde{W}^\varepsilon$. From condition $5$ it follows that both $D_\lambda g$ and $g$ are uniformly bounded on $\tilde{W}^\varepsilon \times X$. We need to show that we can take the norm $\norm{\cdot}_{\mathbb{R}^m}$ such that both terms of \cref{NormParameterNonLinearEq5} are sufficiently small. As we mentioned in \cref{RemarkBoundednessNorm}, the operator norm of $D_\lambda g$ depends on the norm on $\mathbb{R}^m$. Recall that our norm is defined as $\norm{\sum_{i=1} \mu_i e_i}_{\mathbb{R}^m} = M \sum_{i=1}^m |\mu_i|$, then there exists a constant $\mathcal{C}$ independent of $M$ such that
\begin{align*}
\operatornorm{D_\lambda g(\lambda,x)} = \frac{\mathcal{C}}{M}.
\end{align*}
Likewise, we have that the operator norm of  $\frac{\partial \xi}{\partial \lambda}$ scales as $1/M$. Furthermore, we remark that the supremum norm of $g$ does not depend on the norm on $\mathbb{R}^m$. Hence we can make \cref{NormParameterNonLinearEq5} as small as desired if we take $M$ sufficiently large. Therefore, we have a constant $\delta(M) \downarrow 0$ as $M \to \infty$ such that 
\begin{align*}
 \sup_{\substack{\lambda \in \tilde{W}^\varepsilon \\ x \in X }} \sup_{\norm{\mu}_{\mathbb{R}^m} \le 1 } \left\|   B(\lambda,x) \mu \right\|_X \le \delta(M).
\end{align*}
Simultaneously we have that  $1 + M^{-1} \mathcal{C}_x \downarrow 1 $ as $M \to \infty$. Together with the strict upper bound $\sup_{\lambda \in \tilde{W}^{\varepsilon}} \operatornorm{D_x g_\lambda} < L_g$, we see that we can choose $M$ large enough such that \cref{NormParameterNonLinearEq3} becomes
\begin{align*}
\supnorm{D\tilde{g}}  \le \sup_{\substack{\lambda \in \tilde{W}^\varepsilon \\ x \in X }} \sup_{\norm{\mu}_{\mathbb{R}^m} \le 1 } \left\|   B(\lambda,x) \mu \right\|_X + \sup_{\lambda \in \tilde{W}^{\varepsilon}} \operatornorm{D_x g_\lambda} (1 + M^{-1} \mathcal{C}_x )< L_g.
\end{align*}
 In a similar fashion, we can choose $M$ large enough such that $\supnorm{D\tilde{k}_c } < L_c$. In particular, we see that assumption 4 of \cref{MainTheorem} is satisfied for $\tilde{F}{}{}$.
\ \\
\textbf{Verifying the conjugacy equation:} We can apply \cref{MainTheorem} to the dynamical system $\tilde{F}{}{}$. Thus we find $\tilde{K} : Y_c \to Y$ and $\tilde{R} : Y_c \to Y_c$ such that
\begin{align}
\tilde{F}{}{} \circ \tilde{K} = \tilde{K} \circ \tilde{R}, \label{NormParametersEquivalentEq3}
\end{align}
and $\tilde{K} = \iota + \tmatrix{\tilde{k}_c}{\tilde{k}_u}{\tilde{k}_s}$ and $\tilde{R} = \tilde{A}_c + \tilde{r}$. However, we claimed that for $\lambda \in \tilde{W}$ we have
\begin{align*}
( A + g_\lambda) \circ K_\lambda = K_\lambda \circ (A_c + r_\lambda)
\end{align*}
which we still have to prove. So let $\lambda = \sum_{i=1}^m \lambda_i e_i \in \tilde{W}$ and $x \in X_c$. Recall that $\functional{\lambda} = \sum_{i=1}^m \lambda_i x_i$ lies in $X_u \oplus X_s$, thus we have $\functionalsub{u}{\lambda} \in X_u$ and $\functionalsub{s}{\lambda} \in X_s$ such that $\functional{\lambda}  = (\functionalsub{u}{\lambda} , \functionalsub{s}{\lambda} )$. Then we have
\begin{align*}
\tilde{F}{}{} (\tilde{K} (\lambda,x + \functional{\lambda}) )&= \tilde{F}{}{}   \begin{pmatrix}  \lambda \\ x +  k_{c,\lambda}(x) \\  \functionalsub{u}{\lambda} + \tilde{k}_u(\lambda,x + \functional{\lambda} )\\ \functionalsub{s}{\lambda} + \tilde{k}_s(\lambda,x + \functional{\lambda} ) \end{pmatrix}.
\end{align*}
Here we used that $\tilde{k}_c(\lambda,x + \functional{\lambda} ) = (0, k_{c,\lambda}(x))$.  We recall the definition of $\tilde{F}(\lambda, x) =  ( \lambda, F_\lambda(x))$, and we define $k_{u,\lambda}(x) \isdef \functionalsub{u}{\lambda} + \tilde{k}_u(\lambda, x +\functional{\lambda} )$ and $k_{s,\lambda}(x) \isdef \functionalsub{s}{\lambda} + \tilde{k}_s(\lambda,x +  \functional{\lambda}  )$. Then we have
\begin{align} 
\tilde{F}{}{} (\tilde{K} (\lambda,x + \functional{\lambda}) )& = \begin{pmatrix} \lambda \\ F_\lambda{}{} \begin{pmatrix}  x + k_{c,\lambda}(x) \\ k_{u,\lambda}(x) \\ k_{s,\lambda}(x) \end{pmatrix} \end{pmatrix}. \label{NormParametersEquivalentEq1}
\end{align}
On the other hand, we write $y(\lambda) = \sum_{i=1}^m \lambda_i y_i \in X_c$ and obtain
\begin{align}
\tilde{R}(\lambda,x + x(\lambda)) &=  \begin{pmatrix} \lambda + \tilde{r}_{\mathbb{R}^m}(\lambda,x + \functional{\lambda})  \\ A_c x + \functional{\lambda} + y(\lambda) + \tilde{r}_{X_c}(\lambda, x + \functional{\lambda}) \end{pmatrix} \in Y_c . \label{NormParametersEquivalentEq4}
\end{align}
From the $\mathbb{R}^m$ component of \cref{NormParametersEquivalentEq3,NormParametersEquivalentEq1,NormParametersEquivalentEq4} we obtain 
\begin{align*}
\lambda = \lambda + \tilde{r}_{\mathbb{R}^m}(\lambda,x + \functional{\lambda}),
\end{align*}
thus $\tilde{r}_{\mathbb{R}^m}(\lambda,x + \functional{\lambda}) = 0$ and $R_\lambda(x) \isdef A_c x +  y(\lambda) + \tilde{r}_{X_c}(\lambda,x + \functional{\lambda}) \in X_c$. Then \cref{NormParametersEquivalentEq3} becomes
\begin{align}
 \begin{pmatrix} \lambda \\ F_\lambda{}{} \begin{pmatrix}  x + k_{c,\lambda}(x) \\ k_{u,\lambda}(x) \\ k_{s,\lambda}(x) \end{pmatrix} \end{pmatrix} &=  \tilde{F}{}{} (\tilde{K} (\lambda,x + \functional{\lambda})) \nonumber\\
 							&= \tilde{K}(\tilde{R}(\lambda,x + \functional{\lambda})) \nonumber \\
 							&= \tilde{K} \begin{pmatrix} \lambda \\ R_\lambda(x) + \functional{\lambda} \end{pmatrix} \nonumber \\
 							&=  \begin{pmatrix}  \lambda \\ R_{\lambda}(x) +  k_{c,\lambda}(R_\lambda(x)) \\   \functionalsub{u}{\lambda}+ \tilde{k}_u(\lambda,R_\lambda(x) +\functional{\lambda})\\ \functionalsub{s}{\lambda} + \tilde{k}_s(\lambda,R_\lambda(x)+\functional{\lambda}) \end{pmatrix}. \label{NormParametersEquivalentEq2}
\end{align}
With the definition of $k_{u,\lambda}$ and $k_{s,\lambda}$, \cref{NormParametersEquivalentEq2} restricted to $X$ becomes 
\begin{align*}
  F_\lambda{}{} \begin{pmatrix}  x + k_{c,\lambda}(x) \\ k_{u,\lambda}(x) \\ k_{s,\lambda}(x) \end{pmatrix} &=  \begin{pmatrix}   R_{\lambda}(x) +  k_{c,\lambda}(R_\lambda(x)) \\  k_{u,\lambda}(R_\lambda(x))\\ k_{s,\lambda}(R_\lambda(x)) \end{pmatrix} = \left(\iota + \begin{pmatrix} k_{c,\lambda} \\ k_{u,\lambda} \\ k_{s,\lambda} \end{pmatrix} \right) \circ R_\lambda(x).
\end{align*}
Thus \cref{ConjugacyEquationParameter} holds for all $\lambda \in \tilde{W}$. Furthermore, the smoothness of $K_\lambda$ and $R_\lambda$ follows from the smoothness of $\tilde{K}$ and $\tilde{R}$. Finally, $R_\lambda$ is invertible since $\tilde{R}$ is invertible.
\end{proof}

As a consequence of \cref{ParametersTheorem}, we can prove the existence of bifurcations in dynamical systems with parameters. Since the solution of \cref{ConjugacyEquationParameter} is unique, the Taylor expansion of $R(\lambda,x)$ is also unique and determined by \cref{ConjugacyEquationParameter} and our choice of $k_c$. Thus if we choose $k_c$ such that the Taylor approximation of $R$ is the normal form of a bifurcation, we know qualitatively how the dynamical behaviour of $F$ changes on the center manifold as the parameter changes. To obtain quantitative information of the dynamical behaviour on the center manifold, we want local bounds on the Taylor approximation of $R$. That is, given the conjugate dynamical system $R$ and its Taylor expansion $P_R$, we are interested in local bounds on $R - P_R$.

\begin{proposition}\label{TaylorBoundsProp}
Let $X = \mathbb{R}^m$ and $F : X \to X$ a $C^n$, $n \ge 2$, discrete dynamical system on $X$ such that $0$ is a fixed point of $F$. Let $k_c :X_c \to X_c$ be $C^n$ and such that \cref{MainTheorem} holds. Let $P_R$ and $P_K$ be the Taylor expansion of $R$ and $K$ of order $n-1$. Then there exist a neighborhood $\mathcal{U}$ of $0$ such that for every neighborhood $0 \in U \subset \mathcal{U}$ there exists constants $C_R$ and $C_K$ such that 
\begin{align*}
\norm{R(x) - P_R(x)} \le C_R \norm{x}^n && \text{ and } && \norm{K(x) - P_K(x) } \le C_K \norm{x}^n
\end{align*}
for all $x \in U$.
\end{proposition}
\begin{remark}
We first remark that our constants $C_R$ and $C_K$ only depend on the neighborhood $U$, the dynamical system $F$ and our choice of $k_c$. Furthermore, the constants $C_R$ and $C_K$ can  explicitly be computed, see \cref{NewRemark}.
\end{remark}
\begin{remark}
For a dynamical system $F: W \times \mathbb{R}^m \to \mathbb{R}^m$ with parameter space $W \subset \mathbb{R}^k$, we can apply \cref{TaylorBoundsProp} to the extended dynamical system $\tilde{F} : \mathbb{R}^{m+k} \to \mathbb{R}^{m+k}$, where $\tilde{F}$ is defined as in the proof of \cref{ParametersTheorem}. Alternatively, for fixed $\lambda \in \tilde{W}$, we can apply \cref{TaylorBoundsProp} to \cref{ConjugacyEquationParameter}.
\end{remark}

\begin{proof}
Before we start the proof, let us recall what it means that $k_c$ is chosen such that \cref{MainTheorem} holds, which can be found in Remark 2.4 from \cite{Hetebrij19}. First, we have bounds $L_c$, $L_u$ and $L_s$ for  the derivatives of $k_c$, $k_u$ and $k_s$ respectively. We note that $L_u , L_s \le 1 + L_c$. Furthermore, we  have that $R = A_c + r$ is invertible with inverse $T = A_c^{-1} + t$, and the derivatives of $R$ and $T$ are bounded by $\operatornorm{A_c} + L_r$ and $L_{-1}$ respectively. Finally, we have the inequalities
\begin{align*}
\operatornorm{A_u^{-1}} \left( \left( \operatornorm{A_c } + L_r \right)^{n} + L_g + L_u \right) &< 1, \\
L_{-1}^{n} \left( \operatornorm{A_s} \left( 1 + L_{-1}L_s \right) + L_g \left( 1 + L_{-1} \left( 1 + L_c \right) \right) \right) &< 1,
\end{align*}
where $L_g$ is an upper bound for the derivative of $g$. Hence in particular we have the weaker bounds $\operatornorm{A_u^{-1}}  \left( \operatornorm{A_c } + L_r \right)^{n}  < 1$ and $L_{-1}^{n}  \operatornorm{A_s} < 1$.

As $F$ is $C^n$, we write $F =A + P_F + h_F$, where $A + P_F$ is the Taylor expansion up to order $n-1$ of $F$, and $h_F$ is of order $n$ around $0$. Likewise, we write $K = \iota + P_K + h_K$ and $R = A_c + P_R + h_R$. We expand the conjugacy equation using the Taylor expansions to obtain
\begin{align*}
Ah_K(x) - \iota h_R(x) + Q(x,h_K(x),h_R(x)) + h_F(K(x)) - h_K(R(x)) = 0.
\end{align*}
Here we define $Q \isdef AP_K + P_F(\iota + P_K + h_K)  - \iota P_R - P_K(A_c + P_R + h_R)$. We write the equation component-wise, i.e.\ using the splitting $X = X_c \oplus X_u \oplus X_s$ we obtain the three equations
\begin{align}
\left\{
\begin{aligned}
h_R &= A_c h_{K,c} + Q_{1,c} + Q_{2,c} h_{K} + Q_{3,c} h_R + h_{F,c}(K) - h_{K,c}(R), \\
A_u h_{K,u} &=  Q_{1,u} + Q_{2,u} h_{K} + Q_{3,u} h_R + h_{F,u}(K) - h_{K,u}(R), \\
A_s h_{K,s} &=  Q_{1,s} + Q_{2,s} h_{K} + Q_{3,s} h_R + h_{F,s}(K) - h_{K,s}(R), 
\end{aligned}
\right. \label{TaylorBoundsEq1}
\end{align}
where we write $Q$ as a polynomial in $x$, $h_K$ and $h_R$
\begin{align*}
Q(x,h_K(x),h_R(x)) = Q_1(x) + Q_2(x,h_K(x))h_K(x) + Q_3(x,h_R(x))h_R(x).
\end{align*}
We note that $Q_1$ is of order $\norm{x}^n$ at the origin. Furthermore, we have that $A_u$ and $R$ are both invertible, and recall that we defined $T = R^{-1}$. Hence we rewrite the second and third equation of \cref{TaylorBoundsEq1} as 
\begin{align}
\left\{ \begin{aligned}
h_R &= A_c h_{K,c} + Q_{1,c} + Q_{2,c} h_{K} + Q_{3,c} h_R + h_{F,c}(K) - h_{K,c}(R), \\
h_{K,u} &=  A_u^{-1} \left( Q_{1,u} + Q_{2,u} h_{K} + Q_{3,u} h_R + h_{F,u}(K) - h_{K,u}(R) \right), \\
h_{K,s} &=  \left( Q_{1,s} + Q_{2,s} h_{K} + Q_{3,s} h_R + h_{F,s}(K) - A_s h_{K,s} \right) \circ T.
\end{aligned} \right. \label{TaylorBoundsEq2}
\end{align}
To obtain local bounds on $h_R$ and $h_K$, we consider a bounded neighborhood $0 \in U $. Then we can  find intervals $[a_i , b_i]$ such that $U \subset \bigtimes_{i=1}^m [a_i, b_i] \isfed \mathcal{I}_U$. We consider cut-off functions of the form
\begin{align*}
\varphi_a^b(x) = \begin{cases} a & x \le a \\
x &a \le x \le b \\
b & b \le x
\end{cases}
\end{align*}
which we  will generalize to a $C^n$ cut-off function in \cref{ExampleSection}. We then consider the cut-off function on $\mathbb{R}^m$ defined by 
\begin{align}
\varphi_U(x)_i  \isdef \varphi_{a_i}^{b_i} (x_i). \label{Cut-off-function}
\end{align} 
Thus in particular it follows from \cref{TaylorBoundsEq2} that, where we will write $\varphi$ instead of $\varphi_U$,
\begin{align}
\left\{ \begin{aligned}
h_R \circ \varphi &= \left( A_c h_{K,c} + Q_{1,c} + Q_{2,c} h_{K} + Q_{3,c} h_R + h_{F,c}(K) - h_{K,c}(R) \right) \circ \varphi, \\
h_{K,u}  \circ \varphi&=  A_u^{-1} \left( Q_{1,u} + Q_{2,u} h_{K} + Q_{3,u} h_R + h_{F,u}(K) - h_{K,u}(R) \right) \circ \varphi, \\
h_{K,s}   \circ \varphi&=  \left( Q_{1,s} + Q_{2,s} h_{K} + Q_{3,s} h_R + h_{F,s}(K) - A_s h_{K,s} \right) \circ T \circ \varphi.
\end{aligned} \right. \label{TaylorBoundsEq3}
\end{align}
Here we define $Q_2 \circ \varphi (x) \isdef Q_2(\varphi(x), h_K(\varphi(x))) h_K(\varphi(x))$, and likewise define $Q_3 \circ \varphi$. We make the following observations:
\begin{itemize}
\item Any bound we find on $h_R \circ \varphi$ and $h_K \circ \varphi$ will be a local bound for $h_R$ and $h_K$ on $U$.
\item The functions $h_{K,c/s/u} \circ \varphi$, $h_R \circ \varphi$ and $h_F \circ \varphi$ are bounded and of order $\norm{x}^n$ at the origin. Therefore, there exists constants $C_{K,c/s/u}$, $C_R$ and $C_F$ depending on $U$ such that $\norm{h_{K,c/s/u}(\varphi(x))} \le C_{K,c/s/u} \norm{x}^n$, $\norm{h_R(\varphi(x))} \le C_R \norm{X}^n$ and $\norm{h_F(\varphi(x))} \le C_F \norm{X}^n$.
\item The derivatives of $K$, $R$ and $T$ are bounded by $1 +L_c$, $\operatornorm{A_c} + L_r$ and $L_{-1}$ respectively. Hence we have the bounds $\norm{K(x)} \le (1+ L_c ) \norm{x}$, $\norm{R(x)} \le (\operatornorm{A_c} + L_r ) \norm{x}$ and $\norm{T(x)} \le L_{-1} \norm{x}$.
\item There exists a constant $C_Q$ depending on $U$ such that $Q_2 \circ \varphi$ and $Q_3 \circ \varphi$ are bounded by $C_Q$. In particular, the constant $C_Q$ goes to zero when the diameter of the neighborhood $U$ goes to zero.
\item The polynomial $Q_1 \circ \varphi$ is bounded and is of order $\norm{x}^n$ at the origin. Therefore, there exists a constant $C_P$ depending on $U$ such that $\norm{Q_1(\varphi(x))} \le C_P \norm{x}^n$.
\end{itemize}
We can estimate $h_R \circ \varphi$, $h_{K,u}  \circ \varphi$ and $h_{K,s}   \circ \varphi$ using those observations and \cref{TaylorBoundsEq3}. We apply the triangle inequality and substitute the estimates from our observations in \cref{TaylorBoundsEq3} to obtain
\begin{align*}
\left\{ \begin{aligned}
\norm{h_R(\varphi(x))} &\le  ( \operatornorm{A_c} C_{K,c}  +C_P  + C_Q\left( C_{K,c}  + C_{K,s} + C_{K,u} + C_{R} \right) \\
				&\quad \quad  + C_F (1 + L_c)^n + C_{K,c} (\operatornorm{A_c} + L_r)^n ) \norm{x}^n, \\
\norm{h_{K,u} ( \varphi(x))}&\le  \operatornorm{A_u^{-1}} ( C_P +C_Q\left( C_{K,c}  + C_{K,s} + C_{K,u} + C_{R} \right) \\
				&\quad \quad  + C_F (1 + L_c)^n  + C_{K,u} (\operatornorm{A_c} + L_r)^n ) \norm{x}^n , \\
\norm{h_{K,s}  ( \varphi(x))} &\le L_{-1}^n  ( C_P  + C_Q( C_{K,c}  + C_{K,s} + C_{K,u} + C_{R})  \\
				&\quad \quad + C_F(1 + L_c)^n + \operatornorm{A_s} C_{K,s} ) \norm{x}^n.
\end{aligned} \right.
\end{align*}
Here we used the crude estimate $\norm{h_K(\varphi(x))} \le C_{K,c} \norm{x}^n + C_{K,s} \norm{x}^n + C_{K,u} \norm{x}^n$.

To obtain the constants $C_R$, $C_{K,u}$ and $C_{K,s}$ we will show that the map
\begin{align}
\begin{pmatrix}
C_R \\
C_{K,u} \\
C_{K,s}
\end{pmatrix} &\mapsto
\begin{pmatrix*}[l]
 \operatornorm{A_c} C_{K,c}  +C_P  + C_Q\left( C_{K,c}  + C_{K,s} + C_{K,u} + C_{R} \right) \\
				\qquad \qquad \qquad \qquad  + C_F (1 + L_c)^n + C_{K,c} (\operatornorm{A_c} + L_r)^n   \\
 \operatornorm{A_u^{-1}} ( C_P +C_Q\left( C_{K,c}  + C_{K,s} + C_{K,u} + C_{R} \right) \\
				\qquad \qquad \qquad \qquad  + C_F (1 + L_c)^n  + C_{K,u} (\operatornorm{A_c} + L_r)^n ) \\
 L_{-1}^n  ( C_P  + C_Q( C_{K,c}  + C_{K,s} + C_{K,u} + C_{R})  \\
				\qquad \qquad \qquad \qquad + C_F(1 + L_c)^n + \operatornorm{A_s} C_{K,s} ) 
\end{pmatrix*}  \label{TaylorBoundsEq4} 
\end{align}
has a unique component-wise positive fixed point.  We first note that \cref{TaylorBoundsEq4} is of the form $C \mapsto \mathcal{A}C + b$. We find
\begin{align}
\norm{\mathcal{A}}_{\infty} &\le \max \left\{ \operatornorm{A_u^{-1}}(\operatornorm{A_c} + L_r)^n , L_{-1}^n \operatornorm{A_s} \right \}   +3  \max\{1,\operatornorm{A_u^{-1}} , L_{-1}^n \}  C_Q. \label{TaylorBoundsEq5}
\end{align} 
As we mentioned at the beginning of the proof, we have $\operatornorm{A_u^{-1}}(\operatornorm{A_c} + L_r)^n < 1$ and $L_{-1}^n \operatornorm{A_s} < 1$. Furthermore, $C_Q$ can be made as small as desired by decreasing the diameter of $U$. Hence there exists a neighborhood $\mathcal{U}$ of $0$ such that $C \mapsto \mathcal{A}C + b$ is a contraction for every neighborhood $U \subset \mathcal{U}$. We know from the smoothness of $R$ and $K$ that there exist large, non-explicit bounds  $C_R$, $C_{K,u}$ and $C_{K,s}$. From this starting point, by iterating the inequalities above we conclude that the unique fixed point of \cref{TaylorBoundsEq4} provides an explicit bound.
\end{proof}

\begin{remark}\label{NewRemark} We again write $C \mapsto \mathcal{A} C + b$ for \cref{TaylorBoundsEq4}, and note that $\mathcal{A}$ and $b$ are both component-wise positive. In practice, we do not compute the right hand side of \cref{TaylorBoundsEq5} and show that it is strictly less than $1$. Instead, we find $\tilde{C}$ component-wise positive such that $\mathcal{A} \tilde{C} + b < \tilde{C}$ component-wise. From the Min-max Collatz–Wielandt formula it follows that \cref{TaylorBoundsEq4} is a contraction, so it has a unique fixed point, explicitly given by $C  \isdef (\mathcal{A} - \operatorname{Id})^{-1}b$. We can then either use $C$ or our Ansatz $\tilde{C}$, which is component-wise larger than $C$, to obtain the constants $C_R$, $C_{K,u}$ and $C_{K,s}$.
\end{remark}
\begin{remark}\label{TaylorBoundsDerivativeRemark}
In the proof \cref{TaylorBoundsProp} we find constants $C_R$ and $C_K$ such that $\norm{h_R(x)} \le C_R \norm{x}^n$ and $\norm{h_K(x)} \le C_K \norm{x}^n$ for $h_R \isdef R - A_c - P_R$ and $h_K \isdef K - \iota - P_K$. We note that under the assumptions of \cref{TaylorBoundsProp} we can also find constants $C_{R,m}$ and $C_{K,m}$ such that  $\norm{D^m h_R(x)} \le C_{R,m} \norm{x}^{n-m}$ and $\norm{D^m h_K(x)} \le C_{K,m} \norm{x}^{n-m}$. To obtain those bounds, we have to modify the proof slightly. Before we introduce the cut-off function $\varphi_U$ in \cref{Cut-off-function}, we take the $m^{\text{th}}$ derivative on both sides of \cref{TaylorBoundsEq2}. We then obtain an expression similar to \cref{TaylorBoundsEq3} for the composition of the $m^{\text{th}}$ derivative of $h$ and $\varphi_U$, except the right hand side of the equation will be more complicated due to taking derivatives of compositions and products. We can still use the triangle inequality to obtain a system like the system in \cref{TaylorBoundsEq4} we have to solve to obtain $C_{R,m}$ and $C_{K,m}$. That is, the right hand side of the system will be of the form
\begin{align*}
\left(
\begin{pmatrix}
0 & 0 & 0 \\
0 & \operatornorm{A_u^{-1}} (\operatornorm{A_c} + L_r)^{n-m} & 0 \\
0 & 0 & \operatornorm{A_s}  L_{-1}^{n-m}  
\end{pmatrix} + \mathcal{E}_U \right) \begin{pmatrix}
C_{R,m} \\
C_{K,m,u} \\
C_{K,m,s}
\end{pmatrix} + \mathcal{D}_U.
\end{align*}
Here the norm of the matrix $\mathcal{E}_U$ goes to $0$ as the diameter of $U$ goes to zero, and the vector $\mathcal{D}_U$ is continuous as function of the diameter of $U$. Since both $\operatornorm{A_u^{-1}} (\operatornorm{A_c} + L_r)^{n-m} < 1$ and $\operatornorm{A_s} L_{-1}^{n-m} < 1$, there exists a neighborhood $\mathcal{U}$ of $0$ such that  \cref{TaylorBoundsEq4} has a positive solution for $C_{K,m,u}$, $C_{K,m,s}$ and $C_{R,m}$  for every neighborhood $U \subset \mathcal{U}$.

 Finally, if $X$ is higher dimensional, we do not have to assume that our Taylor expansions are of the same order in every variable. For example, we may also obtain a Taylor expansion in two variables $x$ and $y$ of the form $c_{0,0} + c_{1,0} x + c_{2,0}x^2  + c_{3,0} x^3 + c_{0,1}y + c_{1,1} x y$. In that case our bound on $R-P_R$ will reflect the different orders of the Taylor expansion, i.e.\ in our example the bound will be of the form $c_{4,0}x^4 + c_{0,2} y^2 + c_{2,1}x^2 y + c_{1,2} x y^2$.
\end{remark}

\begin{remark}\label{TaylorBoundsBetterRemark}
In the proof of \cref{TaylorBoundsProp} we used the crude estimate $\norm{R (x)} \le (\operatornorm{A_c} + L_r) \norm{x}$ in the observations after \cref{TaylorBoundsEq3}. The bound we find for $C_R$ and $C_K$ by solving \cref{TaylorBoundsEq4} depends on this estimate. That is, if we have a better estimate for $R(x)$, the bound we find for $C_R$ and $C_K$ improves. Since we know the Taylor expansion of $R$, and we know that the constant $C_R$ exists on $U$ as $R$ is $C^n$, we can improve the bound on $R(x)$ using
\begin{align*}
\norm{R(\varphi(x))} = \norm{A_c \varphi(x) + P_R(\varphi(x)) + h_R(\varphi(x))} \le \norm{A_c \varphi(x) + P_R(\varphi(x))}  +C_R \norm{x}^n.
\end{align*}
We can similarly improve the bounds for $\norm{K(x)}$ and $\norm{T(x)}$ on $U$ in the same observation.
\end{remark}

\section{Continuous time dynamical systems}\label{SectionODE}

Besides center manifolds for discrete time systems, we are also interested in center manifolds for ODEs. Hence our second generalization of \cref{MainTheorem} will be a parameterization theorem for ODEs. To prove the existence of  center manifolds, we will use the \tmaps\ of the ODE. That is, we will show that there exists a time-independent conjugacy between the \tmaps\  of the ODE and \tmaps\  of a conjugate vector field on the center subspace. To do so, the \tmaps\  of our ODE have to exist and have to be sufficiently smooth. We first state an existence result. We recall that  we defined sectorial operators in \cref{SectionNotation} such that $\exp(A t)$ exists for all $t \ge 0$ if $A$ is sectorial.

\begin{proposition}\label{FlowExistence}
Let $A: X \to X$ be a sectorial operator and $g : X \to X$ a uniformly bounded $C^n$ function. Then there exists a jointly $C^n$ flow map $\varphi : (0, \infty) \times X \to X$ such that
\begin{align*}
\varphi(t,x) = \exp(A t) x + \int_0^t \exp(A(t-s)) g(\varphi(s,x)) \text{d}s.
\end{align*}
In particular, $\varphi(t,x)$ is a solution to the ODE
\begin{align*}
\dot{x}(t) = Ax(t) + g(x(t)) && \text{ for } t>0.
\end{align*}
\end{proposition}
\begin{proof}
The existence of the flow map for all $t > 0$ is Theorem 3.3.4 of \cite{Henry93} and the smoothness of the flow map is Corollary 3.4.6 of \cite{Henry93}.
\end{proof}
\begin{remark}
The flow $\varphi$ exists for all $t \ge 0$, but if $A$ is not bounded then $\varphi$ is not necessarily differentiable at $t = 0$.
\end{remark}

Beside the existence result of the \tmap, we also need that the non-linearity $G_t$ of the \tmap\ is uniformly bounded. In fact, we have an explicit bound on the derivative of $G_t$ in terms of $A$, $t$ and $\supnorm{Dg}$.

\begin{proposition}\label{DerivativeNonLinearSpace}
Let $A: X \to X$ be a sectorial operator and $g : X \to X$ a uniformly bounded $C^n$ function. We can write the flow map $\varphi_\tau$ as
\begin{align*}
\varphi_\tau(x) = \exp(A \tau) x + G_\tau(x).
\end{align*}
Furthermore, if $g(0) = 0$ and $Dg(0) =0$, then 
\begin{align*}
G_\tau \in \left\{ h \in C_b^n(X) \ \middle| \ h(0) = 0, \ Dh(0) = 0 \text{ and } \supnorm{Dh} \le L_G(\supnorm{Dg}, \tau) \right\},
\end{align*}
where we define 
\begin{align}
L_G(\supnorm{Dg}, \tau) \isdef \tau \supnorm{Dg} \sup_{s \in [ 0, \tau]} \operatornorm{\exp(A s)}^2 \exp\left( \supnorm{Dg} \int_0^\tau \operatornorm{ \exp(A ( \tau -s ))} \ds  \right). \label{definitionLGEquation}
\end{align}
\end{proposition}
\begin{proof}
From \cref{FlowExistence}, we know that the \taumap\ satisfies
\begin{align*}
\varphi_\tau(x)  = \exp(A \tau) x + \int_0^\tau \exp(A(\tau-s)) g(\varphi(s,x)) \text{d}s.
\end{align*}
Thus $G_\tau$ is defined as
\begin{align}
G_\tau (x) \isdef  \int_0^\tau \exp(A(\tau-s)) g( \exp(A s)x + G_s(x)) \text{d}s, \label{DefinitionNonLinearity}
\end{align}
which is jointly $C^n$ on $(0,\infty)$. We will use Gr\"{o}nwall's inequality to show that $G_\tau(0) = 0$. Using the smoothness of $G_\tau$, we can show that also $DG_\tau(0) = 0$. Finally, we will only prove that $\supnorm{DG_\tau} \le L_G( \supnorm{Dg},\tau)$, as the proof that $G_\tau$ is uniformly bounded  in the $C^n$ norm is similar.

By taking the norm on both sides of \cref{DefinitionNonLinearity} for $x = 0$, we obtain
\begin{align*}
\norm{G_\tau(0)} \le \int_0^\tau \operatornorm{ \exp(A (\tau -s ))} \norm{g(G_s(0))} \ds.
\end{align*}
Both $s \mapsto  \norm{g(G_s(0))}$ and $s \mapsto \operatornorm{ \exp(A (\tau -s ))}$ are continuous, thus with Gr\"{o}nwall's inequality we obtain $\norm{G_\tau(0)} \le 0$, and therefore $G_\tau(0) = 0$. 

To show that $DG_\tau(0) = 0$, we take the derivative of \cref{DefinitionNonLinearity} and obtain
\begin{align}
DG_\tau(x) &= \int_0^\tau \exp(A(\tau - s)) Dg(\exp(As)x + G_s(x)) \exp(As) \ds  \nonumber \\
			&\quad + \int_0^\tau \exp(A(\tau - s)) Dg(\exp(As)x + G_s(x)) DG_s (x) \ds. \label{DerivativeNonLinearityODE}
\end{align}
For $x= 0 $ we obtain
\begin{align*}
DG_\tau(0) &= \int_0^\tau \exp(A(\tau - s)) Dg(0) \left( \exp(As) + DG_s(0) \right) \ds  = 0,
\end{align*}
since $Dg(0) = 0$.

For the upper bound on $DG_\tau$, we take the norm on both sides of \cref{DerivativeNonLinearityODE} and obtain
\begin{align*}
\norm{DG_\tau(x)} &\le \int_0^\tau \operatornorm{\exp(A ( \tau -s ))} \supnorm{Dg} \operatornorm{\exp(As)} \ds  \\
				&\quad + \int_0^\tau \operatornorm{ \exp(A ( \tau -s ))} \supnorm{Dg} \norm{DG_s(x)} \ds \\
				&\le \tau \supnorm{Dg} \sup_{s \in [ 0, \tau]} \operatornorm{\exp(A s)}^2 + \int_0^\tau \operatornorm{ \exp(A ( \tau -s ))} \supnorm{Dg} \norm{DG_s(x)} \ds.
\end{align*}
Gr\"{o}nwall's inequality then gives the upper bound
\begin{align*}
\norm{DG_\tau(x)} \le \tau \supnorm{Dg} \sup_{s \in [ 0, \tau]} \operatornorm{\exp(A s)}^2 \exp\left( \supnorm{Dg} \int_0^\tau \operatornorm{ \exp(A ( \tau -s ))} \ds  \right).
\end{align*}
Since the right hand side does not depend on $x$, we find
\begin{align*}
\supnorm{DG_\tau} &\le \tau \supnorm{Dg} \sup_{s \in [ 0, \tau]} \operatornorm{\exp(A s)}^2 \exp\left( \supnorm{Dg} \int_0^\tau \operatornorm{ \exp(A ( \tau -s ))} \ds  \right) \\
				&\isfed L_G(\supnorm{Dg},\tau).
\end{align*}
We can use analogous estimates to recursively show that $G_\tau$ is uniformly bounded in $C^n$ norm. That is, we take the $m$-th derivative of \cref{DefinitionNonLinearity} for $m \le n$, take the norm on both sides of this new equation and apply Gr\"{o}nwall's inequality.
\end{proof}

\subsection{Parameterization theorem for ODEs}

The previous two propositions allow us to use \cref{MainTheorem} on a single \tmap, provided the non-linearity $g$ is small enough. We would like to apply \cref{ParametersTheorem} to a collection of \tmaps. If we consider the dynamical system $(t,x) \mapsto \varphi_{t+\tau}(x)$ on $\mathbb{R} \times X$ as we would do in \cref{ParametersTheorem}, then the non-linearity of our dynamical system is given by $(t,x) \mapsto  \left(\exp(A t) - \exp(A \tau)  \right) x + G_{\tau + t}(x)$. In particular, we see that the time derivative of the non-linearity will be unbounded, hence we can not apply \cref{ParametersTheorem} to a collection of \tmaps.  However, we can construct a collection of discrete dynamical systems on the center subspace, and prove that these discrete dynamical systems are precisely the \tmaps\ of a conjugate vector field on the center subspace.

\begin{theorem}[Center manifolds for flows]\label{ODETheorem} Let $X$ be a Banach space and $f{}{} = A + g$ a vector field on $X$. Here $A : \mathcal{D}(A) \subset X \to X$ is a sectorial operator and $g{}{}: X \to X$ is a  $C^n$, $n \ge 2$, vector field such that $g(0) = 0$ and $Dg(0) = 0$. Let $k_c : X_c \to X_c$ be chosen. Assume that there exists a $\tau > 0$ such that 
\begin{enumerate}
\item There exists closed subspace $X_c$, $X_u$ and $X_s$ such that $X = X_c \oplus X_u \oplus X_s$ for which we have $A : \mathcal{D}(A) \cap X_c \to X_c$, $A : \mathcal{D}(A) \cap X_u \to X_u$ and $A : \mathcal{D}(A) \cap X_s \to X_s$. We write $A = \tmatrix{A_c & 0 & 0}{0 & A_u & 0}{0 & 0 & A_s}$ where we define $A_c \isdef A \big|_{\mathcal{D}(A) \cap X_c}$, and similarly define $A_u$ and $A_s$. In particular, we then have 
\begin{align*}
\exp(A\tau) = \begin{pmatrix}\exp(A_c\tau) & 0 & 0 \\ 0 & \exp(A_u\tau) & 0 \\ 0 & 0 & \exp(A_s \tau)\end{pmatrix}.
\end{align*}
\item The linear operators $\exp(A_c\tau)$ and $\exp(A_u\tau)$ are invertible.
\item The norm on $X$ is such that
\begin{align*}
\operatornorm{\exp(A_c\tau)^{-1} }^{\tilde{n}} \operatornorm{ \exp(A_s\tau) } < 1 && \text{ and } && \operatornorm{\exp(A_u\tau)^{-1} } \operatornorm{\exp(A_c\tau)}^{\tilde{n}} < 1
\end{align*}
for all $1 \le \tilde{n} \le n$.
\item The non-linearities $g$ and $k_c$ satisfy
\begin{align*}
g{}{} &\in \left\{ h \in C^{n}_b(X,X) \mid h(0)=0, \   Dh(0) = 0 \text{ and }\supnorm{Dh} < L_g \right\}, \\ 
k_c &\in \left\{ h \in C^{n+1}_b(X_c,X_c) \mid h(0) = 0 , \ Dh(0) = 0 \text{ and } \supnorm{Dh} < L_c \right\},
\end{align*}
such that $L_G(L_g,\tau)$, defined in \cref{definitionLGEquation}, and $L_c$ are small enough in the sense of \cref{MainTheorem}.
\end{enumerate}
Then there exist a $C^{n}$ conjugacy $K{}{} :  X_c \to X$ and $C^{n}$ vector field $R = A_c + r : X_c \to X_c$ such that the flow $\psi : [0,\infty) \times X_c \to X_c $ of $R$ is mapped to the flow of $f{}{}$ under $K$:
\begin{align*}
\varphi_t \circ K &= K \circ \psi_t \qquad  \text{for all $t \ge 0$.}
\end{align*}
Furthermore, $K{}{}$ and $R$ have the following properties:
\begin{enumerate}[label=\Alph*)]
\item\label{PropertiesR} The vector field $R = A_c + r$ has forward and backward flow which satisfies the ODE $\dot{y}(t) = R(y(t))$ for all initial conditions $y(0) = y_0 \in X_c$. Furthermore, we have
\begin{align*}
r &\in \left\{ h \in C^{n}_b(X_c,X_c) \mid h(0)=0 \text{ and }   Dh(0) = 0  \right\}. 
\end{align*}
\item\label{PropertiesK} The conjugacy $K{}{}$ is given by
\begin{align*}
K{}{} = \iota +  \begin{pmatrix}  k_c \\ k_u \\ k_s \end{pmatrix}
\end{align*}
with $\iota: X_c \to X$ the inclusion map and
\begin{align*}
k_u &\in \left\{ h \in C^{n}_b(X_c,X_u) \mid h(0)=0, \   Dh(0) = 0 \text{ and }\supnorm{Dh} \le L_u \right\}, \\ 
k_s &\in \left\{ h \in C^{n}_b(X_c,X_s) \mid h(0)=0, \   Dh(0) = 0 \text{ and }\supnorm{Dh} \le L_s \right\}. 
\end{align*}
The constants $L_u$ and $L_s$ depend on $L_G(L_g,\tau)$ and $L_c$.
\end{enumerate}
\end{theorem}

\begin{remark} If we assume that $k_c$ is $C^n$ instead of $C^{n+1}$, then the vector field $R$ on $X_c$ would only be $C^{n-1}$.
\end{remark}
\begin{remark}
We note that we again assume that the non-linearity $g$ is bounded in $C^n$. We want to remark that this is a bound on the non-linearity of the ODE $\dot{x} = Ax + g(x)$, and not on the non-linearity of the \taumap.
\end{remark}
\begin{remark} Similar to how we generalized \cref{MainTheorem} to dynamical systems with parameters in \cref{ParametersTheorem}, we can generalize \cref{ODETheorem} to ODEs with parameters. That is, we expand our ODE to include the parameters as variables, i.e., we consider $(\dot{\lambda}, \dot{x}) = (0, Ax + C\lambda + g(\lambda,x))$. Then the proof would be a combination of the proofs of \cref{ParametersTheorem,ODETheorem}.
\end{remark}

\subsection{Obtaining the conjugacy}

As we mentioned before, we want to construct the \tmaps\ of a conjugate vector field on the center subspace. Instead of giving a single long proof of \cref{ODETheorem}, we will split it into several lemmas from which \cref{ODETheorem} will follow. As a starting point, we assume in \cref{ODETheorem} that for a specific time $\tau>0$ the \taumap\ satisfies the conditions of \cref{MainTheorem}. From now on, $\tau > 0$ will always refer to the specific time from \cref{ODETheorem}.

\begin{lemma}\label{Lemma1}
Under the conditions of \cref{ODETheorem}, there exists a $C^n$ conjugacy $K:X_c \to X$ and $C^n$ discrete dynamical system $\Psi_\tau = \exp(A_c \tau) + \psi_\tau : X_c \to X_c$ such that
\begin{align*}
\varphi_\tau \circ K = K \circ \Psi_\tau.
\end{align*}
Furthermore, $K$ and $\Psi_\tau$ have the following properties:
\begin{enumerate}[label=\Alph*)]
\item The dynamical system $\Psi_\tau =\exp(A_c \tau) + \psi_\tau$ is globally invertible and 
\begin{align*}
\psi_\tau &\in \left\{ h \in C^{n}_b(X_c,X_c) \mid h(0)=0 , \    Dh(0) = 0 , \text{ and } \supnorm{Dh} \le L_r \right\}. 
\end{align*}
\item The conjugacy $K{}{}$ is given by
\begin{align*}
K{}{} = \iota +  \begin{pmatrix}  k_c \\ k_u \\ k_s \end{pmatrix}
\end{align*}
with $\iota: X_c \to X$ the inclusion map and
\begin{align*}
k_u &\in \left\{ h \in C^{n}_b(X_c,X_u) \mid h(0)=0 , \ Dh(0) = 0 \text{ and }  \supnorm{Dh} < L_u \right\}, \\ 
k_s &\in \left\{ h \in C^{n}_b(X_c,X_s) \mid h(0)=0 , \      Dh(0) = 0 \text{ and } \supnorm{Dh} < L_s\right\}. 
\end{align*}
\end{enumerate}
\end{lemma}
\begin{proof}
From \cref{FlowExistence,DerivativeNonLinearSpace} and the assumptions of \cref{ODETheorem} it follows that $\varphi_\tau$ together with $k_c$ satisfy the assumptions of \cref{MainTheorem}, from which the lemma follows.
\end{proof}

As mentioned already, we want to construct the \tmaps\ of a conjugate vector field on the center subspace. For the \tmap\ of the vector field $f$ on $X$, we will prove the existence of a conjugate discrete dynamical system $\Psi_t$ on the center subspace. We will then show that the dynamical systems $\Psi_t$ are in fact the \tmaps\ of a conjugate vector field on the center subspace.

\begin{lemma}\label{Lemma2}
Under the conditions of \cref{ODETheorem}, there exists a $T > 0$ such that for all $t \in [0,T)$ there exists a conjugate dynamical system $\Psi_t :X_c \to X_c$ such that $\varphi_t \circ K = K \circ \Psi_t$.
\end{lemma}

Before we prove the lemma, we will give the idea behind the construction of the conjugate dynamical systems $\Psi_t$. We want to solve the conjugacy equation $\varphi_t \circ K = K \circ \Psi_t$ where the conjugacy $K$ is given by the equation $\varphi_\tau \circ K = K \circ \Psi_\tau$. With the flow property of $\varphi_t$, we have
\begin{align*}
\varphi_\tau \circ (\varphi_t \circ K) = \varphi_t \circ ( \varphi_\tau \circ K) =( \varphi_t \circ K) \circ \Psi_\tau.
\end{align*}
Hence we see that $\varphi_t \circ K$ is also a conjugacy between $\varphi_\tau$ and $\Psi_\tau$. In \cite{Hetebrij19} we proved that the center manifold is unique under some mild conditions. In particular, we have that for any two conjugacies $K$ and $\tilde{K}$ there exists a diffeomorphism $\zeta :X_c \to X_c$ such that $K  \circ \zeta = \tilde{K}$. In our case, this means that we obtain diffeomorphisms $\zeta_t$ such that $ K \circ \zeta_t = \varphi_t \circ K$. Thus we have to prove that the conditions which imply a unique center manifold are satisfied. For the sake of completeness, we will repeat the uniqueness results from \cite{Hetebrij19}.

\begin{lemma}[Lemma 6.2 in \cite{Hetebrij19}]\label{Prop62} Let $F{}{}:X \to X$ and $k_c : X_c \to X_c$ satisfy the conditions of \cref{MainTheorem}. Then the conjugacy equation
\begin{align*}
F{}{} \circ \begin{pmatrix}\operatorname{Id} + k_c \\ k_u \\ k_s \end{pmatrix} = \begin{pmatrix} \operatorname{Id} + k_c \\ k_u \\ k_s \end{pmatrix} \circ (A_c + r) 
\end{align*}
 has a unique solution for $k_u \in C^0_b(X_c, X_u)$, $k_s \in C^0_b(X_c,X_s)$ and $r \in C^0_b(X_c, X_c)$  with the property that $A_c + r$ is a homeomorphism.
 \end{lemma}
\begin{lemma}[Proposition 6.3 in \cite{Hetebrij19}]\label{Prop63} Let $F{}{}:X \to X$ and $k_c, \tilde{k}_c : X_c \to X_c$ satisfy the conditions  of \cref{MainTheorem}. Then the image of $K{}{} = \iota + \tmatrix{ k_c}{k_u}{k_s}$ and $\tilde{K{}{}} = \iota +  \tmatrix{ \tilde{k}_c}{\tilde{k}_u}{\tilde{k}_s}$ are the same, for $K{}{},\tilde{K{}{}}$ the (unique) conjugacy obtained from \cref{MainTheorem}. In particular, there exists a diffeomorphism $\zeta : X_c \to X_c$ such that $K \circ \zeta = \tilde{K}$.
\end{lemma}
We are now ready to prove \cref{Lemma2}.
\begin{proof}[Proof of \cref{Lemma2}]
 We first note that from the invertiblity of $\exp(A_c \tau)$ it follows that $A_c$ is bounded. We have
\begin{align*}
A_c = \left( A_c \exp(A_c \tau) \right) \left( \exp(A_c \tau) \right)^{-1}.
\end{align*}
By definition, $\left( \exp(A_c \tau) \right)^{-1}$ is bounded and $A_c \exp(A_c \tau)$ is the time derivative of $t \mapsto \exp(A_c t)$ at $\tau > 0$, hence bounded as $A_c$ is a sectioral operator. So the right hand side is a bounded operator, thus $A_c$ is also a bounded operator. In particular, this means that $\exp(A_c t)$ exists for all $t\in \mathbb{R}$.

For $t \ge 0$ we define the conjugacy  $K_t \isdef \varphi_t \circ K \circ \exp(-A_c t) : X_c \to X$ and dynamical system $R_t \isdef \exp(A_c t) \Psi_\tau \circ \exp(-A_c t) : X_c \to X_c$. From the flow property, $\varphi_{t+s} = \varphi_t \circ \varphi_s$ it follows that
\begin{align*}
\varphi_\tau \circ K_t &= \varphi_\tau \circ \left( \varphi_t \circ K \circ \exp(-A_c t) \right) \\
				 &= \varphi_t \circ \varphi_\tau \circ  K \circ \exp(-A_c t) \\
				&= \varphi_t \circ K \circ \Psi_\tau \circ \exp(-A_c t) \\
				&= \left( \varphi_t \circ K \circ  \exp(-A_c t) \right) \circ \left( \exp(A_c t) \Psi_\tau \circ \exp(-A_c t) \right) \\
				&= K_t \circ R_t.
\end{align*}

As we already mentioned, we want to show that there exists a diffeomorphism  $\zeta_t :X_c \to X_c$ such that $\varphi_t \circ K \circ \exp(- A_c t) = K \circ \zeta_t$, or equivalent
\begin{align*}
\varphi_t \circ K = K \circ \zeta_t \circ \exp(A_c t).
\end{align*}
This would follow directly from \cref{Prop63} if its condition is satisfied. That is, $K_t$ should be the conjugacy obtained from \cref{MainTheorem} when we choose 
\begin{align*}
\tilde{k}_c \isdef  K_{t,c} - \operatorname{Id} =   \left( \exp(A_c t) + G_{c,t} \right) \circ  k_c \circ \exp(-A_c t) + G_{c,t} \circ \exp(-A_c t)
\end{align*}
instead of $k_c$ in \cref{MainTheorem}. We note that we need to check that $\tilde{k}_c$ satisfies the conditions of \cref{MainTheorem}, i.e.\ $\tilde{k}_c$ lies in $C_b^n$ and its derivative is bounded by $L_c$. We will then use \cref{Prop62} to show that $K_t$ is the conjugacy obtained by \cref{MainTheorem}, and hence we can apply \cref{Prop63}. Indeed, we have to show that 
\begin{enumerate}
\item The non-linearity $\tilde{k}_c$ lies in $C_b^n$ and its derivative is bounded by $L_c$.
\item The (un)stable parts of $K_t$ are bounded.
\item The linear part of $R_t$ is $\exp(A_c \tau)$, its non-linearity is bounded, and $R_t$ is a homeomorphism.
\end{enumerate}
\textbf{1)} We have
\begin{align*}
\tilde{k}_c = \left( \exp(A_c t) + G_{c,t} \right) \circ k_c \circ \exp(-A_ct) + G_{c,t} \circ \exp(-A_c t).
\end{align*}
Since all terms of the right hand side are $C^n$, we have that  $\tilde{k}_c$ is $C^n$. Furthermore, $\exp(A_ct)$ is a bounded linear operator, $k_c$ and $G_{c,t}$ are bounded functions in $C^n$, hence $\tilde{k}_c$ is bounded in $C^n$ norm. Finally, its derivative is given by
\begin{align*}
D\tilde{k}_c(x) &= \left(  \exp(A_c t)  + DG_{c,t} (k_c(\exp(-A_c t) x)) \right)  Dk_c( \exp(-A_c t) x) \exp(-A_c t)  \\
		& \quad +   DG_{c,t} (\exp(-A_ct )x ) \exp(-A_c t).
\end{align*}
We obtain the estimate
\begin{align*}
\supnorm{ D\tilde{k}_c} &\le \left( \operatornorm{\exp(A_c t)} + \supnorm{DG_t} \right) \supnorm{Dk_c} \operatornorm{\exp(-A_c t)} + \supnorm{DG_t} \operatornorm{\exp(-A_c t)} \\
					&\le \left( \operatornorm{\exp(A_c t)} + L_G(L_g,t)  \right) L_c \operatornorm{\exp(-A_c t)} + L_G(L_g,t) \operatornorm{\exp(-A_c t)}.
\end{align*}
We want to prove that the right hand side is bounded by $L_c$. From the definition of $L_G$ in \cref{DerivativeNonLinearSpace}, it follows that $s \mapsto L_G(L_g,s)$ is continuous and $L_G(L_g,0) = 0$. Furthermore, $s \mapsto \operatornorm{ \exp(A_cs)}$ is continuous for all $ s \in \mathbb{R}$ since $A_c$ is bounded. In particular, we have $\operatornorm{\exp(A_c 0)} = 1$. Thus for $t = 0$ we have $\supnorm{ D\tilde{k}_c} < L_c$ and by continuity there exists an interval $[0,T)$ such that $\supnorm{ D\tilde{k}_c} < L_c$ for all $t \in [0,T)$. \\
\ \\
 \textbf{2)} We note that $k_u$ is bounded and continuous, hence also
\begin{align*}
K_{u,t} = \exp(A_u t) k_u \circ \exp(-A_c t) + G_{u,t} \circ k_u \circ \exp(-A_c t)
\end{align*}
is bounded and continuous, as $\exp(A_u t)$ is a bounded operator and $G_{u,t}$ is bounded. Similar, we can show that $K_{s,t}$ is bounded and continuous. \\
\ \\
 \textbf{3)} From \cref{Lemma1} it follows that $\Psi_\tau = \exp(A_c \tau) + \psi_\tau $ is a homeomorphism, thus $R_t = \exp(A_c t) \Psi_\tau \circ \exp(-A_c t)$ is also a homeomorphism. Furthermore, $\exp( \pm A_c t)$ and $\exp(A_c \tau)$ commute, thus we have 
\begin{align*}
R_t &= \exp(A_c t) \exp(A_c \tau) \exp(-A_c t) + \exp(A_c t ) \psi_\tau \circ \exp(-A_c t) \\
&= \exp(A_c \tau) + \exp(A_c t) \psi_\tau \circ \exp(-A_c t).
\end{align*}
Thus the linear part of $R_t$ is $\exp(A_c \tau)$, and its non-linearity is bounded, since $\exp(A_c t)$ is a bounded linear operator and $\psi_\tau$ is bounded.

Now that we have proven the conditions for both \cref{Prop62,Prop63}, we apply \cref{Prop62,Prop63} to $K_t$ for all $t \in [ 0, T)$. That is, for $t \in [0,T)$ there exists a diffeomorphism $\zeta_t = \operatorname{Id} + \psi_t \circ \exp(- A_c t)$ such that
\begin{align*}
\varphi_t \circ K \circ \exp(-A_c t) = K_t =K \circ \zeta_t = K \circ ( \exp(A_c t) + \psi_t ) \circ \exp(-A_c t).
\end{align*}
Therefore, $K$ is the conjugacy between the \tmap\ $\varphi_t$ on $X$ and the dynamical system $\Psi_t \isdef \exp(A_c t ) + \psi_t$ on the center subspace for any $t \in [0,T)$.
\end{proof}

\subsection{Obtaining the vector field}

We now have a collection $(\Psi_t)_{t \in [0,T)}$ of dynamical systems on the center subspace, where each dynamical system $\Psi_t$ is conjugate with $\varphi_t$. We want to show that there exists a vector field $\dot{x} = A_c x + r(x)$ on the center subspace such that $\Psi_t$ is its \tmap\ for all $ t \in [0,T)$. Hence we first want to show that $t \mapsto \Psi_t$ is differentiable.

\begin{lemma}\label{Lemma3} Under the conditions of \cref{ODETheorem}, the function $(t,x) \mapsto \psi_t(x)$ is a fixed point of
\begin{align}
 \Xi : \Phi \mapsto \exp(A_c t) k_c(x) + G_{c,t}(K(x)) - k_c(\exp(A_c t) x + \Phi_t(x)). \label{SmoothnessConjugateFlow}
\end{align}
Furthermore, if $\Phi$ is differentiable with respect to time and $x \mapsto \frac{ \partial \Phi}{\partial t}(t,x)$ is $C^n$, then $\Xi(\Phi)$ has the same smoothness.
\end{lemma}
\begin{proof}
From \cref{Lemma2} it follows that $\psi_t$ satisfies
\begin{align*}
\begin{pmatrix}
\exp(A_c t) + G_{c,t} \\ \exp(A_u t ) + G_{u,t} \\ \exp(A_s t ) + G_{s,t} 
\end{pmatrix} \circ \begin{pmatrix} \operatorname{Id} + k_c  \\ k_u \\ k_s \end{pmatrix} 
= \begin{pmatrix} \operatorname{Id} + k_c  \\ k_u \\ k_s \end{pmatrix}  \circ ( \exp(A_c t) + \psi_t).
\end{align*}
In particular, from the center part it follows that $\psi_t$ satisfies
\begin{align}
\psi_t(x) = \exp(A_c t) k_c(x) + G_{c,t}(K(x)) - k_c(\exp(A_c t) x + \psi_t(x)) \label{DefinitionNonLinearityConjugateFlow}
\end{align} 
for $ t\in [0,T)$. Hence $\psi : [0,T) \times X_c \to X_c$ is a fixed point of $\Xi$.

Assume that $\Phi$ is differentiable with respect to time and $x \mapsto \frac{ \partial \Phi}{\partial t}(x)$ is $C^n$, then we want to show that all three terms in \cref{SmoothnessConjugateFlow}, namely $\exp(A_c t) k_c(x)$, $G_{c,t}(K(x))$ and $k_c(\exp(A_c t) x + \Phi_t(x))$ have the same smoothness as $\Phi$. \\
\ \\
\textbf{Term 1)}  Since $A_c$ is bounded, we have that $t \mapsto \exp(A_c t)$ is differentiable at all $t \in \mathbb{R}$ with derivative $A_c \exp(A_c t)$. Hence $t \mapsto \exp(A_c t)$ is differentiable with respect to time and its time derivative is $C^n$ as function on $X$.  \\
\ \\
\textbf{Term 2)} For the center part of $G_t$, we have
\begin{align*}
G_{c,t} = \exp(A_c t) \int_0^t \exp(-A_c s) g_c(\exp(As) x + G_s(x)) \ds.
\end{align*}
From \cref{FlowExistence} it follows that $G_{c,t}$ is differentiable at $t >0$ with derivative 
\begin{align*}
\frac{\partial G_{c,s}}{\partial s} \bigg|_{s = t} &= A_c G_{c,t} + g_{c} \circ \varphi_t,
\end{align*}
which is a $C^n$ function $X$. For the derivative at $0$ we derive, where we denote $H(s,x) = \exp(-A_c s) g_c( \exp(As)x + G_s(x))$,
\begin{align}
\frac{G_{c,t}(x) - G_{c,0}(x) - t g_c(x)}{t} &=  \frac{1}{t}\int_0^t \exp(A_c t) H(s,x) - g_c(x) \ds \nonumber \\
								&= \frac{\exp(A_c t) - \operatorname{Id}}{t} \int_0^t H(s,x)\ds   \label{DerivativeConjugateFlowEq1} \\
								&\quad +\frac{1}{t} \int_0^t \exp(-A_c s) g_c(\exp(As) x + G_s(x)) - g_c(x) \ds. \label{DerivativeConjugateFlowEq2}
\end{align}
We want to show that \cref{DerivativeConjugateFlowEq1,DerivativeConjugateFlowEq2} both go to $0$ when $t \downarrow 0$. This shows that $G_{c,t}$ is differentiable at $t=0$ with derivative $g_c$, which is $C^n$. We start with \cref{DerivativeConjugateFlowEq1}. We note that $H(s,x)$ is uniformly bounded on $[0,T] \times X_c$, and we denote its bound by $\mathcal{C}$. Then we obtain
\begin{align*}
\lim_{t \downarrow 0 }  \left\|  \frac{\exp(A_c t) - \operatorname{Id}}{t} \int_0^t H(s,x)\ds  \right\|_{X_c}
					&\le \lim_{t \downarrow 0 }  \frac{1}{t} \operatornorm{ 
					\exp(A_c t) - \operatorname{Id}} \int_0^t \left\| H(s,x)  \right\|_{X_c} \ds \\
					&\le \lim_{t \downarrow 0 }  \frac{1}{t} \operatornorm{ 
					\exp(A_c t) - \operatorname{Id}} \int_0^t \mathcal{C}  \ds \\
					&\le  \lim_{t \downarrow 0 }  \mathcal{C} \operatornorm{ 
					\exp(A_c t) - \operatorname{Id}}  \\
					&= 0.
\end{align*}
Thus \cref{DerivativeConjugateFlowEq1} goes to $0$ if $t \downarrow 0$. 

To analyze \cref{DerivativeConjugateFlowEq2}, fix $x \in X_c$ and let $\varepsilon > 0$. Then there exists a $\delta > 0$ such that $\norm{\exp(-A_c s)g_c(y) - g_c(x)}_{X_c} < \varepsilon$ if $s < \delta$ and $y \in B_\delta(x)$. Furthermore, there exists a $\gamma \in (0,\delta)$ such that $\exp(As)x + G_s(x) \in B_\delta(x)$ for all $s \in [0, \gamma)$ by continuity of $s \mapsto \exp(As)x + G_s(x)$. Hence for $t \in (0, \gamma)$ we have
\begin{align*}
\left\| \frac{1}{t} \int_0^t \exp(-A_c s) g_c(\exp(As) x + G_s(x)) - g_c(x) \ds \right\|_{X_c} \le \frac{1}{t} \int_0^t \varepsilon \ds = \varepsilon.
\end{align*}
Thus also \cref{DerivativeConjugateFlowEq2} goes to $0$ if $t \downarrow 0$. Hence $G_{c,t}$ is also differentiable at $t=0$, and its time derivative $g_c$ at $t=0$ is $C^n$ as function on $X_c$. \\
\ \\
\textbf{Term 3)} Finally, we have assumed that $k_c \in C^{n+1}(X_c, X_c)$, hence
\begin{align*}
\frac{\partial}{\partial t} k_c(\exp(A_c t) x + \Phi_t(x)) = Dk_c(\exp(A_c t) x + \Phi_t(x)) \left( A_c \exp(A_c t) x+ \frac{\partial \Phi_t}{\partial t} (x) \right)
\end{align*}
is $C^n$ as function on $X$.

Thus if $\Phi$ is differentiable with respect to time and $x \mapsto \frac{\partial \Phi}{\partial t}(t,x)$ is $C^n$, then so is $\Xi(\Phi)$.
\end{proof}

\begin{lemma}\label{Lemma4}
Under the conditions of \cref{ODETheorem}, the function $\Xi$ defined in \cref{SmoothnessConjugateFlow} is a contraction on the set
\begin{align*}
\left\{ \Phi \in C_b^0([0,T) \times X_c, X_c) \ \middle| \ \Phi_t(0) = 0  \text{ for all } t \in [ 0,T) \right\}.
\end{align*}
\end{lemma}
\begin{proof}
Recall that $\supnorm{Dk_c} < L_c$, which is small enough in the sense of \cref{MainTheorem}. In particular, this means that $L_c < 1$. The result then follows from the estimate $\norm{\Xi(\Phi) - \Xi(\tilde{\Phi})}_0 \le \supnorm{Dk_c} \supnorm{\Phi - \tilde{\Phi}}$.
\end{proof}

\begin{lemma}\label{Lemma5}
Under the conditions of \cref{ODETheorem}, the fixed point $\Phi_\Xi$ of the contraction $\Xi$ is differentiable with respect to time and $x \mapsto \frac{\partial \Phi}{\partial t}(t,x)$ is $C^n$.
\end{lemma}
\begin{proof}
The argument for the first part of the lemma is analogous to the argument used in the proof of Lemma 3.2 i) in \cite{Hetebrij19}.
\end{proof}

In particular, we have shown in \cref{Lemma3} that $(t,x) \mapsto \psi_t(x)$ is a fixed point of $\Xi$, and hence from \cref{Lemma4} it follows that it is the unique fixed point of $\Xi$. Therefore, it follows from \cref{Lemma5} that $t \mapsto \psi_t$ is differentiable and $x \mapsto \frac{\partial \psi_t}{\partial t}(x)$ is $C^n$. In particular, we can now prove that the \tmap\ of $\dot{x} = \frac{\partial \Psi_t}{\partial t}\big|_{t=0}(x)$ is given by $\Psi_t$ for all $t\in [ 0,T)$.

\begin{lemma}\label{Lemma6}
Under the conditions of \cref{ODETheorem}, the \tmap\ of the vector field $\dot{x} = R(x) \isdef \frac{\partial \Psi_t}{\partial t}\big|_{t=0}(x)$ on the center manifold is the map $\Psi_t$ defined in \cref{Lemma2} for all $t \in [0,T)$.
Moreover, for all $t \ge 0$, we have $\varphi_t \circ K = K \circ \Psi_t$, where $\Psi_t$ is the \tmap\ of $R$.
\end{lemma}
\begin{proof}
It follows from \cref{Lemma3,Lemma4,Lemma5} that $x \mapsto \frac{\partial \varphi_t}{\partial t}(x)$ is well-defined and $C^n$ for all $t \in [0,T)$. Recall that $\Psi_t(x) = \exp(A_c t)x + \psi_t(x)$ and $A_c$ is bounded, hence $x \mapsto \frac{\partial \Psi_t}{\partial t}(x)$ is well-defined and $C^n$ for all $t \in [0,T)$. Thus we can define the $C^n$ map
\begin{align*}
R(x) \isdef  \frac{\partial \Psi_t}{\partial t} \bigg|_{t=0} (x) = A_c x + \frac{\partial \psi_t}{\partial t}(x)\bigg|_{t=0} (x).
\end{align*}
For $x \in X_c$, we want to show that $t \mapsto \Psi_t(x)$ is the orbit of $x$ under $\dot{x} = R(x)$ for $t \in [0,T)$. That is, we have to show that 
\begin{align*}
\frac{\partial }{\partial t} \Psi_t = R( \Psi_t) = \frac{ \partial \Psi_t}{\partial t} \bigg|_{t=0} ( \Psi_t) \qquad  \text{ for all } t \in[0,T).
\end{align*}
Equivalently, we want to show that $\Psi_t$ has the flow property on $[0,T)$, i.e.\ $\Psi_s \circ \Psi_t = \Psi_{t+s}$ for all $t,s \in [0,T)$ such that $t+s \in [0,T)$. Let $t,s \in [0,T)$ such that $t+s \in [0,T)$, then from \cref{Lemma3,Lemma4} it follows that $\psi_{t+s}$ is the unique fixed point of $\Xi$, and hence $\Psi_{t+s} = \exp(A_c (t+s)) + \psi_{t+s}$ is the unique solution of $\varphi_{t+s} \circ K = K \circ \Psi_{t+s}$. On the other hand, from the flow property $\varphi_s \circ \varphi_t = \varphi_{t+s}$ we obtain
\begin{align*}
K \circ \Psi_{t+s} = \varphi_{t+s} \circ K = \varphi_s \circ \varphi_t \circ K = \varphi_s \circ K \circ \Psi_t = K \circ \Psi_s \circ \Psi_t.
\end{align*}
Therefore, we obtain $\Psi_{t+s} = \Psi_s \circ \Psi_t$. Hence, the \tmap\ of $\dot{x} = R(x)$ is given by $\Psi_t$ defined in \cref{Lemma2} for all $t \in [0,T)$.

We still have to show that all orbits in $X_c$ given by $\dot{x} = R(x)$ exist for all $t \ge 0$ and that $\varphi_t \circ K = K \circ \Psi_t$ for all $t \ge 0$, where $\Psi_t$ is the \tmap\ of $R$. Since $\Psi_t : X_c \to X_c$ is well-defined for all $t \in [0,T)$, we can use the flow property of the \tmaps\ of $R$ to see that \tmap\ of $R$ exists for all $t \ge 0$. Furthermore, if $t \ge 0$, then there exists an $N \in \mathbb{N}$ such that $s = t - TN/2 \in [0,T/2)$, hence we have
\begin{align*}
\varphi_t \circ K &= \varphi_{s} \circ \underbrace{\varphi_{T/2} \circ \cdots \circ \varphi_{T/2}}_{N \text{ times}} \circ K = \varphi_{s} \circ \underbrace{\varphi_{T/2} \circ \cdots \circ \varphi_{T/2}}_{N - 1 \text{ times}} \circ K \circ \Psi_{T/2} \\
		&= \cdots =  K \circ \Psi_{s} \circ \underbrace{\Psi_{T/2} \circ \cdots \circ \psi_{T/2}}_{N \text{ times}} = K \circ \Psi_t. \qedhere
\end{align*}
\end{proof}

\subsection{Proof of \texorpdfstring{\cref{ODETheorem}}{Theorem 3.4}}

We can now finish the proof of \cref{ODETheorem}

\begin{proof}[Proof of \cref{ODETheorem}] 
It follows from \cref{Lemma6} that $K :X_c \to X$ maps the flow of $R :X_c \to X_c$ to the flow of $f{}{} : X \to X$. Furthermore, it follows from \cref{Lemma1} that $K$ has the desired properties. It follows from \cref{Lemma5} and \cref{DefinitionNonLinearityConjugateFlow} that $r$ is a fixed point of the contraction
\begin{align*}
\Theta : \xi \mapsto  A_c k_c (x) + A_c g_c(K(x)) - Dk_c(x)(A_c x + \xi(x)).
\end{align*}
Hence $r$ is bounded, and the set $\{ \xi \mid  \xi(0) = 0 \text{ and } D \xi(0) = 0\}$ is invariant under $\Theta$, thus both $r(0) = 0$ and $Dr(0) = 0$. Since $A_c$ is also bounded, we see that $\dot{x} = R(x) = A_c x + r(x)$ has forward and backward flow.
\end{proof}

Since the conjugate vector field $R$ has backward flow, we can show that the vector field $f$ has backward flow on the center manifold.

\begin{corollary} Under the conditions of \cref{ODETheorem}, there exists a forward and backward flow for $\dot{x} = Ax + g(x)$ on $K(X_c)$.
\end{corollary}
\begin{proof}
If $x_0 \in K(X_c)$, there exists an $y_0 \in X_c$ such that $x_0 = K(y_0)$. The forward flow satisfies
\begin{align*}
\varphi_t(x_0) = \varphi_t(K(y_0)) = K(\Psi_t(y_0)),
\end{align*}
thus $K(X_c)$ in invariant under the forward flow of $f{}{}$. Furthermore, we define the backward flow on $K(X_c)$ by 
\begin{align*}
\varphi_{-t}(x_0) = \varphi_{-t} ( K(y_0)) \isdef K ( \Psi_{-t}(y_0)).
\end{align*}
In particular, we have for all $t,s \in \mathbb{R}$
\begin{align*}
\varphi_t \circ \varphi_s  = \varphi_t \circ K \circ \Psi_s = K \circ \Psi_t \circ \Psi_s = K \circ \Psi_{t+s} = \varphi_{t+s}.
\end{align*}
Thus $\varphi_t$ has the flow property on the center manifold for all $t \in \mathbb{R}$, and in particular we have that $\varphi_t$ is the flow of $f{}{}$ for $t \ge 0$, hence $\varphi_t$ is the flow of $f{}{}$ for all $t \in \mathbb{R}$.
\end{proof}

\section{An application}\label{ExampleSection}

As an illustration of the parameterization method, we will prove the existence of a period doubling bifurcation in a two dimensional discrete dynamical system. Furthermore, we will give explicit regions in the phase space in which the period orbit will lie for small parameter values after the bifurcation. Finally, we also show the existence of heteroclinic orbits from the periodic orbit to the stationary point that the periodic orbit bifurcated from.

\subsection{A bifurcation in a in reaction-diffusion system}

We consider a reaction-diffusion equation with transport on the integer lattice given by
\begin{align}
\dot{u}_n = \underbrace{u_{n-1} - 2u_n + u_{n+1}}_{\text{Diffusion}} + \underbrace{(6 + \lambda) u_n + 3 u_n^2 - u_n^3 }_{\text{Reaction}} + \underbrace{ u_{n-1} -   u_n}_{\text{Transport}}. \label{RDT}
\end{align}
Here $n \in \mathbb{Z}$ and $\lambda$ is a bifurcation parameter. The coefficient $6 + \lambda$ in the reaction term is chosen so that a bifurcation happens at $\lambda = 0$. A stationary solution to \cref{RDT} is an orbit of the discrete dynamical system
\begin{align}
\begin{pmatrix}
u_{n-1} \\
u_{n} \\
\end{pmatrix} &\mapsto 
\begin{pmatrix}
u_{n} \\
u_{n+1} \\
\end{pmatrix} 
=
\begin{pmatrix}
u_{n} \\
- (3 + \lambda) u_n - 2 u_{n-1} - 3 u_n^2 + u_n^3
\end{pmatrix}. \label{True RDT-Equation}
\end{align}
We will prove that at $\lambda =0$ and $u=0$ there occurs a period doubling bifurcation in this map. Furthermore, we want to find computationally the parameter range for which we can prove that there exists a period $2$-orbit, and show that for these parameter values, there also exists a heteroclinic orbit between the origin and the period $2$-orbit. Hence we find a family of stationary solutions $(u_n)_{n \in \mathbb{Z}}$ of \cref{RDT} which limits to $0$ as $n \to -\infty$ and limits to a $2$-periodic profile as $n \to \infty$, or vice versa.

For $\lambda = 0$, the linearization of \cref{True RDT-Equation} at the origin is given by
\begin{align*}
\begin{pmatrix}
0 & 1 \\
-2  & -3
\end{pmatrix}
\end{align*}
which has eigenvalues $-1$ and $-2$. We apply the coordinate transformation $(u_{n-1},u_n) \mapsto (x,y) = (-2u_{n-1} - u_n, 2 u_{n-1}  +2 u_n)$ such that the linearization becomes diagonal, and \cref{True RDT-Equation} is equivalent to
\begin{align}
F{}{} : \begin{pmatrix}  x \\ y \end{pmatrix} &\mapsto   
							 \begin{pmatrix}  -1 & 0 \\  0 & -2 \end{pmatrix} \begin{pmatrix} x \\ y \end{pmatrix} + \begin{pmatrix}  g_c(\lambda, x, y) \\ g_u(\lambda , x , y) \end{pmatrix}, \label{RDT-Equation}
\end{align}
where we define
\begin{align*}
g_c(\lambda,x,y) &= - \left( (x+y)^3 - 3 (x+y)^2 - \lambda (x+y) \right),
 \\
g_u(\lambda,x,y) &= 2 \left( (x+y)^3 - 3 (x+y)^2 - \lambda (x+y) \right).
\end{align*}
If we want to apply \cref{ParametersTheorem}, we need to replace $g_c$ and $g_u$ with bounded $C^n$ functions. Assume for now that we replaced $g_c$ by $h_c$ such that $h_c$ is a uniformly bounded $C^n$ function with $h_c \equiv g_c$ on a neighborhood around $0$. Likewise, assume that $g_u$ is replaced by $h_u$ such that $h_u$ is uniformly bounded in $C^n$ and $h_u \equiv g_u$ on the same neighborhood around $0$. Let us denote this neighborhood by $\mathcal{B}$. Then we consider the dynamical system
\begin{align}
\tilde{F} : \begin{pmatrix} x \\ y \end{pmatrix} \mapsto \begin{pmatrix} -1 & 0 \\ 0 & - 2 \end{pmatrix} \begin{pmatrix} x \\y \end{pmatrix} + \begin{pmatrix} h_c(\lambda,x,y) \\ h_u(\lambda,x,y) \end{pmatrix}. \label{CutOffDynamics}
\end{align}
If we can apply \cref{ParametersTheorem} to \cref{CutOffDynamics} for a suitable choice of $k_c$ and parameter set $W \subset \mathbb{R}$, then the center manifold of \cref{CutOffDynamics} will be a local center manifold for \cref{RDT-Equation}. In fact, the intersection of the center manifold of \cref{CutOffDynamics} and $\mathcal{B}$ is a local center manifold of \cref{RDT-Equation}.

\subsection{Conditions on the linearization}

To apply \cref{ParametersTheorem} to \cref{CutOffDynamics}, we have to choose $k_c$ and a parameter set $W \subset \mathbb{R}$ for which the six conditions of the theorem hold. Conditions $4$ and $6$ depend not only on our choice of $k_c$ and $W$, but also on our choice of the functions $h_c$ and $h_u$, whereas the other conditions do not depend on the choices we make. Therefore, we will check conditions $1$, $2$ and $3$a of \cref{ParametersTheorem} first.

\textit{\underline{Condition 1:}} From the definition of $\tilde{F}$ in \cref{CutOffDynamics}, it follows that 
\begin{align*}
A = \begin{pmatrix} -1 & 0 \\ 0 & -2 \end{pmatrix}.
\end{align*}
Hence we have the invariant subspaces $X_c = \langle e_1 \rangle$ and $X_u = \langle e_2 \rangle$.

\textit{\underline{Condition 2:}} We have $A_c = -1$ and $A_u = -2$, which are both invertible.

\textit{\underline{Condition 3a:}} We have for $n \in \mathbb{N}$
\begin{align*}
\operatornorm{A_u^{-1}} \max\left\{ 1 , \operatornorm{A_c} \right\}^n = \frac{1^n}{2} < 1.
\end{align*}

\subsection{Normal forms}\label{NormalForms}

Before we choose $k_c$, $W$, $h_c$ and $h_u$ and apply \cref{ParametersTheorem} to \cref{CutOffDynamics}, we want to prove that there is  a period doubling bifurcation, and find where the periodic orbit is on the center subspace. This allows us to make an informed choice for $k_c$, $W$, $h_u$ and $h_c$. Assume for now that we can apply \cref{ParametersTheorem}, i.e.\ we have chosen $k_c$, $W$, $h_c$ and $h_u$ all sufficiently small. Then we obtain a conjugate dynamical system
\begin{align*}
R : W \times X_c \to X_c, x \mapsto -x + r_\lambda(x),
\end{align*}
such that $r(0,0) = 0$ and $\frac{\partial r}{\partial x}(0,0) = \frac{\partial r}{\partial \lambda}(0,0) = 0$. The conjugacy we obtain is given by
\begin{align*}
K : W \times X_c \to X, x \mapsto \begin{pmatrix} x + k_c(x) \\ k_{u,\lambda}(x) \end{pmatrix},
\end{align*}
such that $k_u(0,0) = 0$ and $\frac{\partial k_u}{\partial x}(0,0) = \frac{\partial k_u}{\partial \lambda}(0,0) = 0$.

It is clear from the definition of $\tilde{F}$ that the origin is always a fixed point for any $\lambda \in W$. Since we obtain $r$ and $k_u$ as fixed points of a contraction operator $\Theta$, they will lie in any invariant set of $\Theta$. In particular, we have that the set $\left\{ h : \mathbb{R} \times X_c \to \mathbb{R} \times X \ \middle| \ h(\lambda,0) = 0 \text{ for } \lambda \in W \right\}$ is invariant under $\Theta$ as the origin is a fixed point of $\tilde{F}$. Therefore, we have $r(\lambda,0) = 0$ and $k_u(\lambda,0) = 0$ for $\lambda \in W$. To prove the existence of the period-doubling bifurcation, we must have that $R$ is at least $C^3$, and satisfies the conditions
\begin{align*}
\frac{ \partial }{\partial \lambda} \frac{\partial R}{\partial x}(0,0) \neq 0 && \text{ and } && 2 \frac{\partial^3 R}{\partial x^3}(0,0) + 3 \left( \frac{\partial^2 R}{\partial x^2}(0,0)\right)^2 \neq 0,
\end{align*}
see for instance Theorem 4.3 in \cite{Kuznetsov98}.

Since $K$ and $R$ satisfy the conjugacy equation
\begin{align*}
\tilde{F}_\lambda \circ K_\lambda = K_\lambda \circ R_\lambda \qquad \text{for all }\lambda \in W,
\end{align*}
we can compute the higher order derivatives of $r$ and $k_u$ once we have chosen $k_c$. That is, we can solve the conjugacy equation order by order given the Taylor expansion of $k_c$.  In particular,  when we take $k_c$ independent of $\lambda$, we obtain an explicit formula for the partial derivative of $R$ with respect to $x$ at $(\lambda,0)$. We compute
\begin{align*}
\frac{\partial R}{\partial x}(\lambda,0) = -1 + \lambda + \frac{1}{2} \left( - 3 \lambda  - 1 +  \sqrt{ \lambda^2 + 6 \lambda + 1} \right).
\end{align*}
We use Taylor's Theorem to write
\begin{align*}
\sqrt{ \lambda^2 + 6 \lambda + 1} = 1 + 3 \lambda - o(\lambda^2).
\end{align*}
Thus we obtain
\begin{align}
\frac{\partial^2 R}{\partial \lambda \partial x}(0,0) &= \frac{\partial}{\partial \lambda} ( - 1 + \lambda +o(\lambda^2) ) = 1 \neq 0. \label{HigherOrderDerivativesTaylor2}
\end{align}
For the second and third derivative with respect to $x$ of $R$ at $(0,0)$, we write $k_c(x) = c_2 x^2 + o(x^3)$ and compute
\begin{align}
\begin{aligned}
\frac{\partial^2 R}{\partial x^2}(0,0) &= 6 - 4 c_2, \\
\frac{\partial^3 R}{\partial x^3}(0,0) &= -78 + 72 c_2 -24 c_2^2.
\end{aligned}\label{HigherOrderDerivativesTaylor}
\end{align}
For any choice of $c_2$ we have $2 \frac{\partial^3 R}{\partial x^3}(0,0) + 3 \left( \frac{\partial^2 R}{\partial x^2}(0,0)\right)^2 = - 48 \neq 0$. Therefore, \cref{CutOffDynamics} has a period-doubling bifurcation at $\lambda = 0$. However, this gives us no a priori information about the interval of existence or the location of the periodic orbit of \cref{CutOffDynamics}. To obtain this information of the $2$-periodic orbit or to prove the existence of heteroclinic orbits, we need an approximation of  both $R$ and $K$ as well as error bounds on our approximations.

\subsection{Taylor approximations}

As we see in \cref{HigherOrderDerivativesTaylor}, our choice of $k_c$ has influence on both the second and third derivative of $R$ with respect to $x$. The normal form of a period doubling bifurcation is $x \mapsto (-1+\lambda)x - c x^3 + o(x^4)$, so we take $k_c$ such that the second derivative of $R$ with respect to $x$ vanishes in the origin. To make the second derivative of $R$ vanish, i.e.\ $c_2 = \frac{3}{2}$, we take 
\begin{align*}
k_c \equiv \frac{3}{2} x^2
\end{align*}
locally near $0$. From \cref{HigherOrderDerivativesTaylor,HigherOrderDerivativesTaylor2}  we obtain the Taylor polynomial $P_R$ of $R$ around $(0,0)$:
\begin{align*}
P_R(\lambda,x) =  (- 1 + \lambda )x  - 4 x^3.
\end{align*}
We see that $P_R$ has the $2$-periodic orbit $x = \pm \sqrt{\lambda}/2$ for $\lambda > 0$. On the other hand, we could also have chosen $k_c = 0$, in which case the Taylor polynomial of $R$ would be $P(\lambda,x) = (- 1 + \lambda )x + 3 x^2 -13 x^3$. This means that we cannot find an analytic expression for the $2$-periodic orbit near $0$ of $P$. As a consequence, our analysis of the conjugate dynamics becomes harder when we choose $k_c = 0$.
Furthermore, we also compute the Taylor Polynomial $P_K$ of $k_u$ around $(0,0)$ for our choice of $k_c \equiv \frac{3}{2} x^2$ locally around $0$. We obtain
\begin{align*}
P_K(\lambda,x) = -2  \lambda x -2 x^2 + 8 x^3.
\end{align*}

\subsection{Error bounds on the conjugate dynamics}

As $P_R$ is a good approximation for the conjugate dynamics $R$, the periodic orbit of $R$ should lie close to the periodic orbit of $P_R$.  In fact, if we have a bound on the derivative of $R - P_R$, we can explicitly find an interval which contains the periodic orbit of $R$. This is the content of the following proposition.

\begin{proposition}\label{ErrorBoundsR} Let $n=3$ and assume that $k_c$, $W$, $h_c$ and $h_u$ are chosen such that \cref{ParametersTheorem} holds for \cref{CutOffDynamics}. Let $0 \le \mathcal{E}_R < 1$ and denote $\mathcal{I}_\lambda =  [-\sqrt{1 + \mathcal{E}_R}\sqrt{\lambda}/2 , \sqrt{1 + \mathcal{E}_R} \sqrt{\lambda}/2]$.
\begin{itemize}
\item[i)] If for all $x \in \mathcal{I}_\lambda$ we have the estimate $\norm{R(\lambda,x) - P_R(\lambda,x)}_X \le \mathcal{E}_R \lambda |x|$, then $R$ has a $2$-periodic orbit alternating between $W_-$ and $W_+$, where we define
\begin{align*}
W_- &\isdef [- \sqrt{1 + \mathcal{E}_R} \sqrt{\lambda} / 2,- \sqrt{1 - \mathcal{E}_R} \sqrt{\lambda} / 2], \\
W_+ &\isdef [ \sqrt{1 - \mathcal{E}_R} \sqrt{\lambda} / 2, \sqrt{1 + \mathcal{E}_R} \sqrt{\lambda} / 2]. 
\end{align*}
\item[ii)] If we have the stronger estimates $0 \le \mathcal{E}_R < 1/2$ and $\operatornorm{ D_x [ R- P_R](\lambda, x)] } \le \mathcal{E}_R \lambda$ for all $x \in \mathcal{I}_\lambda$, then the periodic orbit alternating between $W_-$ and $W_+$ is unique for $0 \le \lambda <4/3$ . Let us denote the periodic orbit by $x_{\pm} \in W_{\pm}$. Then all points $x \in (x_- , x_+)$ converge to the periodic orbit as $n \to - \infty$ and to the origin as $n \to \infty$.
\end{itemize}
\end{proposition}

\begin{remark}
If we choose $k_c \equiv \frac{3}{2} x^2$ on $\mathcal{I}_\lambda$, then the Taylor Series of $R$ starts with $P_R$, i.e.\ $R(\lambda,x) - P_R(\lambda,x) = O(\lambda^2,\lambda x^2, x^4)$. Furthermore, if $x \in \mathcal{I}_\lambda$, we have that $x$ is of order $\sqrt{\lambda}$. Thus we have that $D_x[R- P_R](\lambda,x) = o(\lambda x , x^3) = o(\lambda^{3/2})$ for $x \in \mathcal{I}_\lambda$. Hence for $\lambda$ sufficiently small we may expect that $\mathcal{E}_R$ exists and is less than $1/2$. That is, we expect that the conditions of \cref{ErrorBoundsR} ii) are satisfied for $\lambda$ sufficiently small.
\end{remark}

\begin{proof}
 \textit{i)} We want to show that $R(W_\pm) \supset W_\mp$, as the existence of a periodic orbit then follows from the Intermediate Value Theorem. We define $\lambda_\pm \isdef \sqrt{1 \pm \mathcal{E}_R} \sqrt{\lambda}/2$.
\begin{enumerate}
\item First, consider the extreme case $R(\lambda,x) = P_R(\lambda,x) - \mathcal{E}_R \lambda x$. Then the periodic orbit of $R$ is $\pm \lambda_-$. This means that the conjugate dynamics $R$ maps $\lambda_-$ to, with $\gamma \in [-\mathcal{E}_R, \mathcal{E}_R]$,
\begin{align*}
R(\lambda,\lambda_-) &= P_R(\lambda,\lambda_-) + \gamma \lambda \lambda_- \\
					&= P_R(\lambda, \lambda_-) - \mathcal{E}_R\lambda \lambda_- +(\gamma + \mathcal{E}_R) \lambda \lambda_- \\
					& =- \lambda_- +(\gamma + \mathcal{E}_R) \lambda \lambda_-.
\end{align*}
As $\gamma + \mathcal{E}_R \ge 0$, we see that $R(\lambda,\lambda_-) \ge - \lambda_-$. Similarly, we have that $R(\lambda,\lambda_+) \le - \lambda_+$. Hence we have $W_- \subset R(\lambda,W_+)$.

\item As a consequence, there exists an interval $U_+ \subset W_+$ such that $R(\lambda,U_+) = W_-$.

\item Similarly to steps 1 and 2, we can find an interval $U_- \subset W_-$ such that $R(\lambda,U_-) = W_+$. Therefore, we can find an interval $V_+ \subset U_+$ such that $R(\lambda,V_+) = U_-$.
\end{enumerate}

Hence, we find an interval $V_+ \subseteq W_+$ such that $R^2(\lambda,V_+) = W_+$. Since $0 \not\in W_+$, $R^2$ has a non-trivial fixed point in $W_+$. Therefore, we see that the $2$-periodic orbit of $R$ alternates between $W_-$ and $W_+$.

\textit{ii)} For the second part of the proposition, we denote a $2$-periodic orbit with $x_\pm \in W_\pm$, i.e.\ $R(\lambda,x_\pm) = x_\mp$. We want to show that for all $x \in (x_-, x_+)$ the orbit of $x$ goes to the origin. First, we will show for all $x \in (-\lambda_-,\lambda_-)$ that the orbits of $x$ goes to the origin. In particular, we will show the stronger result that $|R(\lambda,x)| < |x|$ for any non-zero $x \in (-\lambda_-,\lambda_-)$. Let $\gamma = R(\lambda,x) - P_R(\lambda,x) \in [- \mathcal{E}_R\lambda, \mathcal{E}_R \lambda]$, then we write out what it means that $|R(\lambda,x)| < |x|$, where we distinguish four different cases:
\begin{align*}
\begin{cases} R(\lambda,x) > -x \iff \lambda x - 4 x^3 + \gamma x > 0 \iff \lambda  + \gamma > 4 x^2  \\
R(\lambda,x) < x \iff \lambda x - 4 x^3 + \gamma x < 2 x \iff \lambda  + \gamma < 2 \end{cases}&& \text{ if } x \in (0,\lambda_-), \\
\begin{cases}
R(\lambda,x) < -x \iff \lambda x - 4 x^3 + \gamma x < 0 \iff \lambda  + \gamma > 4 x^2 \\
R(\lambda,x) > x \iff \lambda x - 4 x^3 + \gamma x > 2 x \iff \lambda  + \gamma < 2
\end{cases}&& \text{ if } x \in (-\lambda_-,0).
\end{align*}
We have $\lambda + \gamma \ge \lambda - \mathcal{E}_R \lambda = 4 \lambda_-^2 > 4x^2$ and with $\mathcal{E}_R < 1/2$ and $\lambda < 4/3$ we also have $\lambda + \gamma\le  \lambda + \mathcal{E}_R \lambda < 3/2 \lambda < 2$. Therefore, if $x \in ( - \lambda_-, \lambda_-)$ we have that the orbit of $x$ goes to the origin.

Now, let $x \in (x_-, - \lambda_-]$. It follows from \cref{ParametersTheorem} that $R(\lambda,\cdot)$ is invertible and continuous, hence we have $x_+ = R(\lambda,x_-) > R(\lambda,x) > R(\lambda,0) = 0$. We want to show that $x$ moves away from the periodic orbit under one iteration of $R(\lambda,\cdot)$, i.e.\ we want to show that $x_+ - R(\lambda,x) > x - x_-$. We can enclose
\begin{align*}
D_x R(\lambda,y) \in [-1 -(2 + 4 \mathcal{E}_R) \lambda, - 1 - (2 - 4 \mathcal{E}_R) \lambda] .
\end{align*}
Since $\mathcal{E}_R < 1/2$, we have $D_x R(\lambda,y) < -1$ for all $y \in (x_-,x)$.  We obtain from the mean value theorem, with $y \in (x_-,x)$, that
\begin{align*}
x_+ - R(\lambda,x) = R(\lambda,x_-) - R(\lambda,x) = D_xR(\lambda,y) ( x_- -  x) > x - x_-.
\end{align*}
Therefore, we have that $x$ moves away from the periodic orbit $x_\pm$ after one iteration of $R(\lambda,\cdot)$. Analogous, we can show that for $x \in [\lambda_-,x_+)$ we obtain that $R(\lambda,x) \in ( x_-, 0)$ as well as that $x$ moves away from the periodic orbit $x_\pm$ under one iteration of $R(\lambda, \cdot)$. Together with the fact that all orbits starting in $(-\lambda_-,\lambda_-)$ go towards the origin, we conclude that orbits starting in $(x_-, - \lambda_-]$ and $[ \lambda_-,x_+)$ also go towards the origin. As we already mentioned, the dynamical system $R(\lambda,\cdot)$ is invertible and continuous, thus the orbit of $x \in (x_-,x_+)$ emerges from the periodic orbit. Hence we have shown that all points $x \in (x_-,x_+)$ converge to the periodic orbit as $n \to - \infty$ and to the origin as $n \to \infty$. 

We still have to show that the periodic orbit in $W_\pm$ is unique, so assume that there is another periodic orbit $x'_\pm \in W_\pm$. Then either $x'_- \in (x_-,x_+)$ or $x_- \in (x'_-,x'_+)$. Thus the orbit of either $x'_-$ or $x_-$ goes to the origin, which contradicts that both $x_\pm$ and $x'_\pm$ are $2$-periodic.
\end{proof}

\subsection{Error bounds on the conjugacy}

Similarly to how we found explicit error bounds on the $2$-periodic orbit of $R$ given explicit errors on its Taylor polynomial, we want to find an explicit rectangle in $\mathbb{R}^2$ in which the image of $K$ lies for all $x \in \mathcal{I}_\lambda$ given the Taylor polynomial $P_K$ of $k_u$. These are given by the following proposition.

\begin{proposition}\label{ErrorBoundsK} 
Let $n=3$ and assume that $k_c$, $W$, $h_c$ and $h_u$ are chosen such that \cref{ParametersTheorem} holds for \cref{CutOffDynamics}.  Furthermore, let $\mathcal{E}_R < 1/2$ and $\lambda < 4/3$ such that \cref{ErrorBoundsR} i) holds. Let $0 \le \mathcal{E}_K \le 9/2$ and $0 \le \lambda < 1/43 < 4/3$. If for all $x \in \mathcal{I}_\lambda$ we have the pointwise estimate $\norm{k_u(\lambda,x) - P_K(\lambda,x)}_X \le 2 \mathcal{E}_K \lambda |x|$, then we have
\begin{align}
k_u(\lambda, x) &\le  \frac{ \left(\sqrt{1 + 12 \lambda + 12 \mathcal{E}_K \lambda}-1\right)^2 \left(2 \sqrt{1 + 12 \lambda + 12 \mathcal{E}_K \lambda}  + 1 \right)}{216}, \label{SupKu} \\
 k_u(\lambda, x) &\ge- \frac{\lambda}{2}  \left( 1 + \mathcal{E}_R + 2 \mathcal{E}_R  \sqrt{(1 + \mathcal{E}_R) \lambda } + 2 \mathcal{E}_K \sqrt{(1 + \mathcal{E}_R) \lambda } \right), \label{InfKu}
\end{align} 
for all $x \in \mathcal{I}_\lambda$.
\end{proposition}

\begin{proof}
The computations in this proof are checked symbolically in the Mathematica Notebook available at \cite{Mathematica}. We have for $x \in \mathcal{I}_\lambda$ 
\begin{align*}
P_K(\lambda,x) - 2 \mathcal{E}_K \lambda x  \le k_u(\lambda,x) \le P_K(\lambda,x) + 2  \mathcal{E}_K \lambda x && \text{ if } x \ge 0, \\
P_K(\lambda,x) + 2 \mathcal{E}_K \lambda x  \le k_u(\lambda,x) \le P_K(\lambda,x) - 2  \mathcal{E}_K \lambda x && \text{ if } x \le 0,
\end{align*}
with $P_K(\lambda,x) = -2  \lambda x -2 x^2 + 8 x^3$. Hence we can bound the minimum of $k_u(\lambda,x)$ on $\mathcal{I}_\lambda$ by 
\begin{align}
\min \left\{ \min_{x \in [0,\lambda_+]} P_K(\lambda,x) - 2 \mathcal{E}_K \lambda x , \min_{x \in [-\lambda_+,0]} P_K(\lambda,x) + 2 \mathcal{E}_K \lambda x \right\} \label{MinimumK}.
\end{align}
 Under the constraints $0 \le \mathcal{E}_R < 1/2$, $0 \le \mathcal{E}_K \le 9/2$ and $0 \le \lambda < 1/43$, it turns out that the the minimum in \cref{MinimumK} is obtained at the boundary $x = -\lambda_+$, hence we have the following lower bound on $k_u(\lambda,x)$ for all $x \in \mathcal{I}_\lambda$:
\begin{align*}
k_u(\lambda,x) &\ge P_K(\lambda, - \sqrt{1 + \mathcal{E}_R} \sqrt{\lambda } /2 ) - 2  \mathcal{E}_K \lambda  \sqrt{1 + \mathcal{E}_R} \sqrt{\lambda }/2 \\
			&= - \frac{\lambda}{2}  \left( 1 + \mathcal{E}_R + 2 \mathcal{E}_R  \sqrt{(1 + \mathcal{E}_R) \lambda }+ 2 \mathcal{E}_K \sqrt{(1 + \mathcal{E}_R) \lambda } \right).
\end{align*}
Likewise, we can bound the maximum of $k_u(\lambda,x)$ on $\mathcal{I}_\lambda$ by
\begin{align}
\max \left\{ \max_{x \in [0,\lambda_+]} P_K(\lambda,x) + 2 \mathcal{E}_K \lambda x , \max_{x \in [-\lambda_+,0]} P_K(\lambda,x) - 2 \mathcal{E}_K \lambda x \right\} \label{MaximumK}.
\end{align}
The maximum of \cref{MaximumK} is obtained in $x = 1/12-\sqrt{1 + 12 \lambda + 12 \mathcal{E}_K \lambda } / 12$, hence for $x \in \mathcal{I}_\lambda$ we have the upper bound
\begin{align*}
k_u(\lambda,x) &\le P_K(\lambda, 1/12-\sqrt{1 + 12 \lambda + 12 \mathcal{E}_K \lambda } / 12 ) \nonumber \\
					&\quad - 2 \mathcal{E}_K \lambda \left( 1/12-\sqrt{1 + 12 \lambda + 12 \mathcal{E}_K \lambda } / 12 ) \right) \nonumber \\
					&=   \frac{ \left(\sqrt{1 + 12 \lambda + 12 \mathcal{E}_K \lambda}-1\right)^2 \left(2 \sqrt{1 + 12 \lambda + 12 \mathcal{E}_K \lambda}  + 1 \right)}{216}. \qedhere
\end{align*}
\end{proof}

\begin{corollary}\label{ErrorBoundsKc}
Under the assumptions of \cref{ErrorBoundsK} and the extra assumption that $k_c \equiv \frac{3}{2} x^2$ on $\mathcal{I}_\lambda$, the dynamical system \cref{CutOffDynamics} for parameter value $\lambda$ has a $2$-periodic orbit inside 
\begin{align*}
\mathcal{B}_\lambda = [\lambda_{c,-}, \lambda_{c,+}] \times [ \lambda_{u,-}, \lambda_{u,+}] ,
\end{align*}
where 
\begin{align*}
\lambda_{c,-} &\isdef  - \sqrt{ 1 + \mathcal{E}_R} \sqrt{\lambda}/2 + \frac{3(1 + \mathcal{E}_R) \lambda}{8}, \\
\lambda_{c,+} &\isdef  \sqrt{ 1 + \mathcal{E}_R} \sqrt{\lambda}/2 + \frac{3(1 + \mathcal{E}_R) \lambda}{8}, \\
\lambda_{u,-} &\isdef - \frac{\lambda}{2}  \left( 1 + \mathcal{E}_R + 2 \mathcal{E}_R  \sqrt{(1 + \mathcal{E}_R) \lambda } + 2 \mathcal{E}_K \sqrt{(1 + \mathcal{E}_R) \lambda } \right), \\
\lambda_{u,+} &\isdef \frac{ \left(\sqrt{1 + 12 \lambda + 12 \mathcal{E}_K \lambda}-1\right)^2 \left(2 \sqrt{1 + 12 \lambda + 12 \mathcal{E}_K \lambda}  + 1 \right)}{216}.
\end{align*}
\end{corollary}
\begin{proof}
The first interval of $\mathcal{B}_\lambda$ is the image of $\mathcal{I}_\lambda$ under $\operatorname{Id} + k_c \equiv x + \frac{3}{2}x^2$. The second interval of $\mathcal{B}_\lambda$ follows from \cref{ErrorBoundsK}.
\end{proof}

We note that we can obtain the location of the periodic orbit more precisely inside $\mathcal{B}_\lambda$. Since we know that the periodic orbit of $R$ lies inside $W_\pm$, the periodic orbit of \cref{CutOffDynamics} is contained in the image of $W_\pm$ under $K$. We are however not only interested in the periodic orbit, but also in heteroclinic connections between the origin and the periodic orbit, which is why we consider the image of $\mathcal{I}_\lambda$ instead.

\subsection{Periodic orbits and connections}

From \cref{ErrorBoundsKc} we find the box where the $2$-periodic orbits of \cref{CutOffDynamics} are. To prove the same periodic orbits for \cref{RDT-Equation}, we want that $h_c \equiv g_c$ and $h_u \equiv g_u$ on the boxes $\mathcal{B}_\lambda$ for some parameter interval $[0, \lambda_{\max}]$. Furthermore, we have to check that we can indeed find $\mathcal{E}_R$ and $\mathcal{E}_K$ such that \cref{ErrorBoundsR,ErrorBoundsK} are both satisfied for all $\lambda \in [0,\lambda_{\max}]$.

\begin{theorem}
Consider the dynamical system given by \cref{RDT-Equation}.
\begin{itemize} \item[i)]
The dynamical system undergoes a period doubling bifurcation at $(\lambda,x) = (0,0)$. 
\item[ii)] For $0 < \lambda \le 7.6 \cdot 10^{-5}$ the $2$-periodic orbit of \cref{RDT-Equation} for parameter value $\lambda$ lies inside the box $\mathcal{B}_\lambda$ from \cref{ErrorBoundsKc}, where we take $\mathcal{E}_R = 57.1\sqrt{\lambda} $ and $\mathcal{E}_K = 61.9 \sqrt{\lambda}$. Furthermore, there exists a 1D manifold inside $\mathcal{B}_\lambda$ consisting of heteroclinic connections between the origin and the periodic orbit.
\end{itemize}
\end{theorem}
\begin{remark}
As we already mentioned, we have that the periodic orbit of \cref{RDT-Equation} lies inside the image of $W_\pm$. As $W_\pm \sim \sqrt{\lambda}$ and $k_c = x + \frac{3}{2}x^2$ on $\mathcal{I}_\lambda$, the distance between the origin and the periodic orbit has a magnitude of $10^{-3}$.
\end{remark}
\begin{proof}
We used the Mathematica Notebook available at \cite{Mathematica} to check several inequalities in the proof below. We will replace $g_c$ and $g_u$ with $C^3$ bounded functions $h_c$ and $h_u$. To do this, we want to consider $h_c = g_c \circ \Phi$, where $\Phi : \mathbb{R}^2 \to \mathbb{R}^2$ is a bounded $C^3$ function. Furthermore, we want to construct $\Phi$ such that $\Phi \equiv \operatorname{Id}$ in a neighborhood around the origin. Then $h_c$ is bounded in $C^3$ and $h_c \equiv g_c$ in the neighborhood where $\Phi \equiv \operatorname{Id}$. To define $\Phi$, we define $\varphi_{\alpha_1,\Delta_1}^{\alpha_2,\Delta_2} : \mathbb{R} \to \mathbb{R}$ as
\begin{align*}
\varphi_{\alpha_1,\Delta_1}^{\alpha_2,\Delta_2}(x)  = \begin{cases} \alpha_1 - \frac{1}{2} \Delta_1 & x \le \alpha_1 - \Delta_1 \\
					x + \frac{(x - \alpha_1)^6}{\Delta_1^5}+\frac{3 (x-\alpha_1)^5}{\Delta_1^4}+\frac{5 (x - \alpha_1)^4}{2 \Delta_1 ^3}			& \alpha_1 - \Delta_1 \le x \le \alpha_1 \\
					x & \alpha_1 \le x \le \alpha_2 \\
					x -\frac{(x - \alpha_2)^6}{\Delta_2^5}+\frac{3 (x-\alpha_2)^5}{\Delta_2^4}-\frac{5 (x - \alpha_2 )^4}{2 \Delta_2^3}  	& \alpha_2 \le x \le \alpha_2 + \Delta_2 \\
					\alpha_2 + \frac{1}{2} \Delta_2 & \alpha_2 + \Delta_2 \le x  \end{cases}
\end{align*}
Here we assume that $\Delta_1$ and $\Delta_2$ are both positive. We find that $\varphi_{\alpha_1,\Delta_1}^{\alpha_2,\Delta_2}$ is $C^3$, it image is the interval $[a_1 - \Delta_1/2,a_2 + \Delta_2/2]$ and its derivative lies in the interval $[0,1]$ for all $x \in \mathbb{R}$.  

For \cref{ParametersTheorem} we consider the parameter set $W = [-10^{-6}, 7.61 \cdot 10^{-5}]$. We have to take negative $\lambda$-values in $W$ for \cref{ParametersTheorem} if we want to conclude something for $\lambda =0$.

Let $\lambda_{\max} = 7.6 \cdot 10^{-5}$, $\mathcal{E}_R = 57.1 \sqrt{\lambda_{\max}}$, $\mathcal{E}_K = 61.9 \sqrt{\lambda_{\max}}$ and consider the set $\tilde{W} = [-10^{-7}, \lambda_{\max}] \subset W$. Recall the definition of the interval $\mathcal{I}_\lambda \subset \mathbb{R}$ and $\lambda_+$ from \cref{ErrorBoundsR}, and the definition of the box $\mathcal{B}_\lambda \subset \mathbb{R}^2$ from \cref{ErrorBoundsKc}. Then we use $a_1,a_2,b_1,b_2,c_1,c_2$ to denote $\mathcal{I}_{\lambda_{\max}} = [a_1 , a_2]$ and $\mathcal{B}_{\lambda_{\max}} = [b_1 , b_2] \times [c_1, c_2]$. Furthermore, let $\Delta = 10^{-6}$, which we can choose to be arbitrary small, but is chosen to be $10^{-6}$ so we can compute explicit bounds in \cref{ComputedLGLC} below. We define $h_c$ and $h_u$ as 
\begin{align}
h_c(x,y) \isdef g_c\left( \varphi_{b_1,\Delta}^{b_2, \Delta}(x) , \varphi_{c_1, \Delta}^{c_2, \Delta}(y) \right) &&\text{ and } &&
h_u(x,y) &\isdef g_u\left( \varphi_{b_1,\Delta}^{b_2, \Delta}(x) , \varphi_{c_1, \Delta}^{c_2, \Delta}(y) \right). \label{CutoffEquationAnsatzG} \end{align}
Finally, we define $d_1 = - \lambda_{\max,+} - 2\mathcal{E}_R \lambda_{\max} \lambda_{\max,+}$ and $d_2 = \lambda_{\max,+} - 2\mathcal{E}_R \lambda_{\max} \lambda_{\max,+}$. Then we choose $k_c(x) =  \frac{3}{2} \varphi_{d_1,\Delta}^{d_2, \Delta}(x)^2$, which is also a $C^3$ bounded function. For all $\mu \in W$ we compute, where we use Mathematica to check the inequalities,
\begin{align}
\begin{aligned}
\supnorm{Dh} &\le \sup_{\substack{b_1 - \Delta/2 \le x \le b_2 + \Delta/2 \\ c_1 - \Delta/2 \le y \le c_2 + \Delta/2}}  | 12 (x+y)^2 - 24(x+y) - 4 \mu|  < 0.13, \\
\supnorm{Dk_c} &\le \sup_{d_1 -\Delta/2 \le x \le d_2 + \Delta/2} 3 |x| <0.017.  \end{aligned} \label{ComputedLGLC}
\end{align}
Hence we take $L_g = 0.13$ and $L_c = 0.017$, in which case we find with Mathematica that condition 4 of \cref{MainTheorem} holds for $n=3$. In particular, the conditions of \cref{ParametersTheorem} are satisfied for $n=3$ for the system \cref{CutOffDynamics}, which means that there exists a $C^3$ conjugacy $K : \tilde{W} \times \mathbb{R}$ and $C^3$ dynamical system $R : \tilde{W} \times \mathbb{R} \to \mathbb{R}$ such that locally for all $\mu \in \tilde{W}$
\begin{align*}
F_\mu \circ K_\mu = K_\mu \circ R_\mu.
\end{align*}
We will now prove the two statements of the theorem.

\textit{i)} From \cref{NormalForms} it follows that \cref{RDT-Equation} undergoes a period doubling bifurcation at the origin at $\lambda = 0$.

\textit{ii)} Fix $0 < \lambda \le \lambda_{\max}$. We want to use \cref{TaylorBoundsProp} for the dynamical system $F_\lambda \isdef F(\lambda, \cdot) : \mathbb{R}^2 \to \mathbb{R}^2$ to find explicit error bounds on the Taylor approximations of $R_\lambda \isdef R(\lambda,\cdot)$ and $K_\lambda \isdef K(\lambda,\cdot)$ on the neighborhood $0 \in \mathcal{I}_\lambda = [a_1,a_2]$. In particular, we want to find error bounds on the first derivative, thus we use \cref{TaylorBoundsDerivativeRemark} in combination with the proof of \cref{TaylorBoundsProp}. Furthermore, we estimate both $DR(x)$ and $DK(x)$ using the Taylor approximation to obtain better bounds, as we explained in \cref{TaylorBoundsBetterRemark}. Finally, since we have used the Ansatz that 
\begin{align*}
\norm{DR_\lambda(x) - DP_R(\lambda,x)}  &\le \mathcal{E}_R \lambda,  \\
\norm{Dk_{u,\lambda}(x) - DP_K(\lambda,x)} &\le  2 \mathcal{E}_K \lambda,
\end{align*}
we check that the bound we obtain is consistent with both these inequalities.

We use the Ansatz on $R_\lambda$ in order to guarantee that $R_\lambda(\mathcal{I}_\lambda) \subset [d_1, d_2]$, i.e.\ we have that $k_c = \frac{3}{2} x^2$ on the image of $\mathcal{I}_\lambda$ under $R$. Furthermore, the Ansatz on $K_\lambda$ allows us to use \cref{ErrorBoundsK}, which tells us that $K_\lambda(\mathcal{I}_\lambda) \subset \mathcal{B}_\lambda \subset \mathcal{B}_{\lambda_{\max}}$, i.e.\ the cut-off functions we used in  \cref{CutoffEquationAnsatzG} are the identity on the image of $K_\lambda$. We can now apply \cref{TaylorBoundsProp} in combination with \cref{NewRemark}, where we choose the cut-off function $\varphi_{\mathcal{I}_\lambda}$ for \cref{Cut-off-function}. We refer to our Mathematica supplement for the computation of the (generalization of the) system \cref{TaylorBoundsEq4}, as well as checking that the bounds we obtain from solving this system are indeed at most $\mathcal{E}_R\lambda$ and $2 \mathcal{E}_K\lambda$.

Finally, the assertion follows from \cref{ErrorBoundsKc} since we check with Mathematica that $\mathcal{E}_R  < \frac{1}{2}$, $\mathcal{E}_K < \frac{9}{2}$ and $\lambda < \frac{1}{43}$ and thus the conditions of \cref{ErrorBoundsK}, and hence the conditions of \cref{ErrorBoundsKc}, are satisfied.
\end{proof}

\bibliography{bib}
\end{document}